\documentclass[11pt]{amsart}
\usepackage{amsopn}
\usepackage{amssymb, amscd}
\usepackage{multirow}

\newcommand{\nc}{\newcommand}

\nc{\fg}{\mathfrak{f} } \nc{\vg}{\mathfrak{v} } \nc{\wg}{\mathfrak{w} }
\nc{\zg}{\mathfrak{z} } \nc{\ngo}{\mathfrak{n} } \nc{\kg}{\mathfrak{k} }
\nc{\mg}{\mathfrak{m} } \nc{\bg}{\mathfrak{b} } \nc{\ggo}{\mathfrak{g} }
\nc{\ggob}{\overline{\mathfrak{g}} } \nc{\sog}{\mathfrak{so} }
\nc{\sug}{\mathfrak{su} } \nc{\spg}{\mathfrak{sp} } \nc{\slg}{\mathfrak{sl} }
\nc{\glg}{\mathfrak{gl} } \nc{\cg}{\mathfrak{c} } \nc{\rg}{\mathfrak{r} }
\nc{\hg}{\mathfrak{h} } \nc{\tg}{\mathfrak{t} } \nc{\ug}{\mathfrak{u} }
\nc{\dg}{\mathfrak{d} } \nc{\ag}{\mathfrak{a} } \nc{\pg}{\mathfrak{p} }
\nc{\sg}{\mathfrak{s} } \nc{\affg}{\mathfrak{aff} } \nc{\qg}{\mathfrak{q} }

\nc{\pca}{\mathcal{P}} \nc{\nca}{\mathcal{N}} \nc{\lca}{\mathcal{L}}
\nc{\oca}{\mathcal{O}} \nc{\mca}{\mathcal{M}} \nc{\tca}{\mathcal{T}}
\nc{\aca}{\mathcal{A}} \nc{\cca}{\mathcal{C}} \nc{\gca}{\mathcal{G}}
\nc{\sca}{\mathcal{S}} \nc{\hca}{\mathcal{H}} \nc{\bca}{\mathcal{B}}
\nc{\dca}{\mathcal{D}} \nc{\val}{\operatorname{val}}

\nc{\vp}{\varphi} \nc{\ddt}{\frac{d}{dt}} \nc{\dds}{\frac{d}{ds}}
\nc{\dpar}{\frac{\partial}{\partial t}} \nc{\im}{\mathtt{i}}

\nc{\SO}{\mathrm{SO}} \nc{\Spe}{\mathrm{Sp}} \nc{\Sl}{\mathrm{SL}}
\nc{\SU}{\mathrm{SU}} \nc{\Or}{\mathrm{O}} \nc{\U}{\mathrm{U}} \nc{\Gl}{\mathrm{GL}}
\nc{\Se}{\mathrm{S}} \nc{\Cl}{\mathrm{Cl}} \nc{\Spein}{\mathrm{Spin}}
\nc{\Pin}{\mathrm{Pin}} \nc{\G}{\mathrm{GL}_n(\RR)} \nc{\g}{\mathfrak{gl}_n(\RR)}

\nc{\RR}{{\Bbb R}} \nc{\HH}{{\Bbb H}} \nc{\CC}{{\Bbb C}} \nc{\ZZ}{{\Bbb Z}}
\nc{\FF}{{\Bbb F}} \nc{\NN}{{\Bbb N}} \nc{\QQ}{{\Bbb Q}} \nc{\PP}{{\Bbb P}} \nc{\OO}{{\Bbb O}}

\nc{\vs}{\vspace{.2cm}} \nc{\vsp}{\vspace{1cm}} \nc{\ip}{\langle\cdot,\cdot\rangle}
\nc{\ipp}{(\cdot,\cdot)} \nc{\la}{\langle} \nc{\ra}{\rangle} \nc{\unm}{\tfrac{1}{2}}
\nc{\unc}{\tfrac{1}{4}} \nc{\und}{\tfrac{1}{16}} \nc{\no}{\vs\noindent}
\nc{\lam}{\Lambda^2(\RR^n)^*\otimes\RR^n} \nc{\tangz}{{\rm T}^{\rm Zar}}
\nc{\nor}{{\sf n}}  \nc{\mum}{/\!\!/} \nc{\kir}{/\!\!/\!\!/}
\nc{\Ri}{\tfrac{4\Ric_{\mu}}{||\mu||^2}} \nc{\ds}{\displaystyle}
\nc{\ben}{\begin{enumerate}} \nc{\een}{\end{enumerate}} \nc{\f}{\frac}
\nc{\lb}{[\cdot,\cdot]} \nc{\isn}{\tfrac{1}{||v||^2}}
\nc{\gkp}{(\ggo=\kg\oplus\pg,\ip)} \nc{\ukh}{(\ug=\kg\oplus\hg,\ip)}
\nc{\tgkp}{(\tilde{\ggo}=\kg\oplus\pg,\ip)}
\nc{\wt}{\widetilde} \nc{\mm}{M}

\nc{\Hess}{\operatorname{Hess}} \nc{\ad}{\operatorname{ad}}
\nc{\Ad}{\operatorname{Ad}} \nc{\rank}{\operatorname{rank}}
\nc{\Irr}{\operatorname{Irr}} \nc{\End}{\operatorname{End}}
\nc{\Aut}{\operatorname{Aut}} \nc{\Inn}{\operatorname{Inn}}
\nc{\Der}{\operatorname{Der}} \nc{\Ker}{\operatorname{Ker}}
\nc{\Iso}{\operatorname{I}} \nc{\Diff}{\operatorname{Diff}}
\nc{\Lie}{\operatorname{L}} \nc{\tr}{\operatorname{tr}} \nc{\dif}{\operatorname{d}}
\nc{\sen}{\operatorname{sen}} \nc{\modu}{\operatorname{mod}}
\nc{\CRic}{\operatorname{PP}} \nc{\Cric}{\operatorname{P}} \nc{\Ricci}{\operatorname{Ric}}
\nc{\sym}{\operatorname{sym}} \nc{\herm}{\operatorname{herm}} \nc{\symac}{\operatorname{sym^{ac}}}
\nc{\symc}{\operatorname{sym^{c}}} \nc{\scalar}{\operatorname{sc}}
\nc{\grad}{\operatorname{grad}} \nc{\ricci}{\operatorname{Rc}}
\nc{\Nor}{\operatorname{Norm}}  \nc{\ricc}{\operatorname{Rc^{c}}}
\nc{\Ricc}{\operatorname{Ric^{c}}} \nc{\ricac}{\operatorname{Rc^{ac}}}
\nc{\Ricac}{\operatorname{Ric^{ac}}} \nc{\Riem}{\operatorname{Rm}}
\nc{\riccig}{\operatorname{ric^{\gamma}}} \nc{\Rin}{\operatorname{M}}
\nc{\Le}{\operatorname{L}} \nc{\tang}{\operatorname{T}}
\nc{\level}{\operatorname{level}} \nc{\rad}{\operatorname{r}}
\nc{\abel}{\operatorname{ab}} \nc{\CH}{\operatorname{CH}}
\nc{\mcc}{\operatorname{mcc}} \nc{\Adj}{\operatorname{Adj}}
\nc{\Order}{\operatorname{O}}  \nc{\inj}{\operatorname{inj}} \nc{\proy}{\operatorname{pr}}
\nc{\vol}{\operatorname{vol}} \nc{\Diag}{\operatorname{Diag}}
\nc{\Spec}{\operatorname{Spec}}

\theoremstyle{plain}
\newtheorem{theorem}{Theorem}[section]
\newtheorem{proposition}[theorem]{Proposition}
\newtheorem{corollary}[theorem]{Corollary}
\newtheorem{lemma}[theorem]{Lemma}

\theoremstyle{definition}
\newtheorem{definition}[theorem]{Definition}

\theoremstyle{remark}
\newtheorem{remark}[theorem]{Remark}

\newtheorem{example}[theorem]{Example}

\title[Geometric flows]{Geometric flows and their solitons on homogeneous spaces}

\author{Jorge Lauret}

\address{Universidad Nacional de C\'ordoba, FaMAF and CIEM, 5000 C\'ordoba, Argentina}
\email{lauret@famaf.unc.edu.ar}

\thanks{This research was partially supported by grants from CONICET, FONCYT and SeCyT (Universidad Nacional de C\'ordoba)}

\dedicatory{Dedicated to Sergio.}

\begin{document}

\maketitle

\begin{abstract}
We develop a general approach to study geometric flows on homogeneous spaces.  Our main tool will be a dynamical system defined on the variety of Lie algebras called the {\it bracket flow}, which coincides with the original geometric flow after a natural change of variables.  The advantage of using this method relies on the fact that the possible pointed (or Cheeger-Gromov) limits of solutions, as well as self-similar solutions or soliton structures, can be much better visualized.  The approach has already been worked out in the Ricci flow case and for general curvature flows of almost-hermitian structures on Lie groups.  This paper is intended as an attempt to motivate the use of the method on homogeneous spaces for any flow of geometric structures under minimal natural assumptions. As a novel application, we find a closed $G_2$-structure on a nilpotent Lie group which is an expanding soliton for the Laplacian flow and is not an eigenvector.
\end{abstract}

\tableofcontents

\section{Introduction}\label{intro}

The aim of this work is to develop a general approach to study geometric flows on homogeneous spaces which relies on the variety of Lie algebras.  The approach has been worked out in \cite{spacehm,homRF} for the homogeneous Ricci flow, in \cite{homRS} for Ricci solitons on homogeneous spaces and in \cite{SCF} for general curvature flows of almost-hermitian structures on Lie groups (see Section \ref{exa-sec} for a short overview on more applications).  This paper is intended as an attempt to motivate the use of the method on homogeneous spaces for any geometric evolution under minimal natural assumptions.

We consider a geometric flow on a given differentiable manifold $M$ of the form
$$
\dpar\gamma(t)=q(\gamma(t)),
$$
where $\gamma(t)$ is a one-parameter family of (tensor fields attached to) geometric structures on $M$ and $\gamma\mapsto q(\gamma)$ is an assignment of a tensor field on $M$ of the same type associated to geometric structures of a given class.  Typically $q(\gamma)$ is a curvature tensor, a Laplacian or the gradient field of some natural geometric functional.  Recall that a geometric structure may be defined by a set of tensor fields $\gamma$ (e.g. a almost-hermitian structures), so in that case the geometric flow will consist of a set of differential equations, one for each tensor.  Our basic assumption is that the flow is invariant by diffeomorphisms, i.e. $q(\vp^*\gamma)=\vp^*\gamma$ for any $\vp\in\Diff(M)$.

\begin{remark}\label{rem1}
However, in the case when a complex manifold $(M,J)$ is fixed and a flow for hermitian metrics or any other geometric structure $\gamma$ on $(M,J)$ is to be considered, the tensor $q$ and so the flow will be assumed to be only invariant by bi-holomorphic maps of $(M,J)$ rather than by diffeomorphisms of $M$.  The symplectic analogous assumption will be made for flows of compatible metrics on a fixed symplectic manifold and its symplectomorphisms.
\end{remark}

\subsection{Geometric flows on homogeneous spaces}
On a homogeneous space $M=G/K$, if we fix a reductive (i.e. $\Ad(K)$-invariant) decomposition $\ggo=\kg\oplus\pg$, then any $G$-invariant geometric structure on $M$ is determined by an $\Ad(K)$-invariant tensor $\gamma$ on $\pg\equiv T_oM$.  Therefore, by requiring $G$-invariance of $\gamma(t)$ for all $t$, the flow equation becomes equivalent to an ODE for a one-parameter family $\gamma(t)$ of $\Ad(K)$-invariant tensors on the single vector space $\pg$ of the form
\begin{equation}\label{i-flow}
\ddt \gamma(t) = q(\gamma(t)).
\end{equation}
Thus short-time existence (forward and backward) and uniqueness (among $G$-invariant ones) of solutions are guaranteed.  This is an advantageous feature, as for most of the geometric flows studied in the literature, short-time existence and uniqueness of solutions are still open problems in the noncompact general case.

A second assumption we make on the geometric structure is that for any fixed $\gamma$, the orbit
\begin{equation}\label{i-nondeg3}
\Gl(\pg)\cdot\gamma
\end{equation}
is open in the vector space $T$ of all tensors of the same type as $\gamma$.  Such orbit consists precisely of those tensors which are non-degenerate in some sense.  We note that this holds for many classes of geometric structures, including Riemannian metrics, almost-hermitian and $G_2$ structures (see Example \ref{exa}).  Consider $\theta:\glg(\pg)\longrightarrow \End(T)$, the representation obtained as the derivative of the natural left $\Gl(\pg)$-action on tensors (i.e.\ $\theta(A)\gamma:=\ddt|_0\left(e^{-tA}\right)^*\gamma$).  If $\glg(\pg)=\ggo_\gamma\oplus\qg_\gamma$ is an $\Ad(G_\gamma)$-invariant decomposition, where $G_\gamma\subset\Gl(\pg)$ is the stabilizer subgroup at $\gamma$ and $\ggo_\gamma:=\{ A\in\glg(\pg):\theta(A)\gamma=0\}$ its Lie algebra, then, for each tensor $q\in T$, there exists a unique linear operator
\begin{equation}\label{i-nondeg}
Q:\pg\longrightarrow\pg, \quad Q\in\qg_\gamma, \quad\mbox{such that}\quad q=\theta(Q)\gamma.
\end{equation}

\begin{remark}\label{rem2}
In the complex case (see Remark \ref{rem1}), $\Gl(\pg)$ must be replaced with
$$
\Gl(\pg,J):=\{ h\in\Gl(\pg):hJ=Jh\},
$$
which is isomorphic to $\Gl_n(\CC)$ if $\dim{M}=2n$, and with
$$
\Spe(\pg,\omega):=\{ h\in\Gl(\pg):h^tJh=J\}\simeq\Spe(n,\RR)
$$
in the symplectic case.
\end{remark}

Let $\gamma(t)$ be a $G$-invariant solution on the homogeneous space $M=G/K$ to the geometric flow \eqref{i-flow}, starting at $\gamma:=\gamma(0)$.  Assume that $G$ is simply connected and $K$ connected, so $M$ is simply connected.  Since $\gamma(t)$ is nondegenerate, for each $t$ there exists $h(t)\in\Gl(\pg)$ such that $\gamma(t)=h(t)^*\gamma$ (see  \eqref{i-nondeg3}).  This implies that, for each $t$, there is an equivalence of geometric structures
$$
\vp(t):(M,\gamma(t)) \longrightarrow \left(G_{\mu(t)}/K_{\mu(t)},\gamma\right),
$$
where 
$$
\mu(t):=\tilde{h}(t)\cdot\lb=\tilde{h}(t)[\tilde{h}(t)^{-1}\cdot,\tilde{h}(t)^{-1}\cdot], \qquad
\tilde{h}(t):=\left[\begin{smallmatrix} I&0\\ 0&h(t) \end{smallmatrix}\right]:\ggo\longrightarrow\ggo,
$$
a Lie bracket on the underlying vector space $\ggo$ isomorphic to the Lie bracket $\lb$ of $\ggo$, $G_{\mu(t)}$ is the corresponding simply connected Lie group and $K_{\mu(t)}$ the connected Lie subgroup of $G_{\mu(t)}$ with Lie algebra $\kg$.  Indeed, the equivariant diffeomorphism $\vp(t)$ defined by the Lie group isomorphism $G\longrightarrow G_{\mu(t)}$ with derivative $\tilde{h}(t)$ satisfies that $\gamma(t)=\vp(t)^*\gamma$.  Note that for each $t$,  the homogeneous space $G_{\mu(t)}/K_{\mu(t)}$ is equipped with the $G_{\mu(t)}$-invariant geometric structure determined by the fixed tensor $\gamma$.

A natural question arises: How does the family of Lie brackets $\mu(t)$ evolve?

\subsection{Bracket flow}
It follows from \eqref{i-nondeg} that for each $t$, there exists a unique operator $Q_t\in\qg_{\gamma(t)}$ such that $q(\gamma(t))=\theta(Q_t)\gamma(t)$.  We can now formulate our first main result (see Theorem \ref{BF-thm} for a more complete statement).

\begin{theorem}\label{i-BF}
If $h(t)\in\Gl(\pg)$ solves the ODE $\ddt h(t)=-h(t)Q_t$, $h(0)=I$, then $\gamma(t)=h(t)^*\gamma$ and
\begin{equation}\label{i-BFeq}
\ddt\mu(t)=\delta_{\mu(t)}\left(\left[\begin{smallmatrix} 0&0\\ 0&Q_{\mu(t)} \end{smallmatrix}\right]\right), \qquad\mu(0)=\lb,
\end{equation}
where $Q_\mu\in\qg_{\gamma}$ is the operator defined by $\theta(Q_\mu)\gamma=q(G_\mu/K_\mu,\gamma)$ and $\delta_\mu:\glg(\ggo)\longrightarrow\Lambda^2\ggo^*\otimes\ggo$ is given by
$$
\delta_\mu(A):=\mu(A\cdot,\cdot)+\mu(\cdot,A\cdot)-A\mu(\cdot,\cdot), \qquad\forall A\in\glg(\ggo).
$$
Conversely, if $\mu(t)$ is a solution to \eqref{i-BFeq} and $h(t)\in\Gl(\pg)$ solves the ODE $\ddt h(t)=-Q_{\mu(t)}h(t)$, $h(0)=I$, then $\gamma(t)=h(t)^*\gamma$ and  $\mu(t):=\tilde{h}(t)\cdot\lb$ for all $t$.
\end{theorem}

Evolution equation \eqref{i-BFeq} is called the {\it bracket flow}.  A direct consequence of the theorem is that the geometric flow solution $\gamma(t)$ and the bracket flow solution $\mu(t)$ differ only by pullback by time-dependent diffeomorphisms. Thus the maximal interval of time $(T_-,T_+)$ where a solution exists is the same for both flows, so the bracket flow can be used as a tool to study regularity questions on the flow (see Sections \ref{evol-norm} and \ref{sec-reg}).  We prove for instance that the velocity of the flow $q(\gamma(t))$ must blow up at a finite-time singularity (i.e.\ either $T_+<\infty$ or $T_->-\infty$) for any geometric flow.

The previous theorem has also the following application on convergence, which follows from the discussion given in Section \ref{converg}, based on \cite{spacehm}, on convergence of homogeneous manifolds.  Suppose that the class of geometric structures involved either contains or determines a Riemannian metric $g_\gamma$ for each $\gamma$ (e.g.\ almost-hermitian and $G_2$ structures).

\begin{corollary}\label{i-conv}
Assume that $\mu(t_k)\to\lambda$ for some subsequence of times $t_k\to T_\pm$.
\begin{itemize}
\item[(i)] If there is a positive lower bound for the (Lie) injectivity radii of the $G$-invariant metrics $g_{\gamma(t_k)}$ on $M=G/K$, then, after possibly passing to a subsequence, the Riemannian manifolds $\left(M,g_{\gamma(t_k)}\right)$ converge in the pointed (or Cheeger-Gromov) sense to $(G_\lambda/K_\lambda,g_{\gamma})$, as $k\to\infty$.

\item[(ii)] In the case of a Lie group $M=G$, the hypothesis on the injectivity radii in part (i) can be removed.  Moreover, if either $G_\lambda$ is compact or $G$ is completely solvable, then the geometric structures $(M,\gamma(t_k))$ smoothly converges up to pull-back by diffeomorphisms to $(G_\lambda,\gamma)$, as $k\to\infty$.
\end{itemize}
\end{corollary}

We note that the limiting Lie group $G_\lambda$ in the above corollary might be non-isomorphic to $G$, and consequently in part (i), the limiting homogeneous space $G_\lambda/K_\lambda$ might be non-homeomorphic to $M$.

\subsection{Solitons}
It is well known that a geometric structure $\gamma$ on a differentiable manifold $M$ will flow self-similarly along a geometric flow $\dpar\gamma=q(\gamma)$, in the sense that the solution $\gamma(t)$ starting at $\gamma$ has the form $\gamma(t)=c(t)\vp(t)^*\gamma$, for some $c(t)\in\RR^*$ and $\vp(t)\in\Diff(M)$, if and only if
$$
q(\gamma)=c\gamma+\lca_{X}\gamma, \qquad \mbox{for some}\quad c\in\RR, \quad X\in\chi(M)\; \mbox{(complete)},
$$
where $\lca_X$ denotes Lie derivative.  In analogy to the terminology used in Ricci flow theory, we call such $\gamma$ a {\it soliton geometric structure}.

\begin{remark}\label{rem3}
The diffeomorphisms $\vp(t)$ must be bi-holomorphims (resp. symplectomorphisms) for flows of hermitian (resp. compatible) metrics or any kind of geometric structures $\gamma$ on a fixed complex (resp. symplectic) manifold (see Remark \ref{rem1}).
\end{remark}

On homogeneous spaces, in view of the equivalence between any geometric flow and the corresponding bracket flow given by Theorem \ref{i-BF}, it is natural to also wonder about self-similarity for bracket flow solutions.  This leads us to consider Lie brackets which only evolves by scaling:  $\mu(t)=c(t)\cdot\lb$, for some $c(t)\in\RR^*$ (see \eqref{scmu}).  Our second main result shows that this gives rise to soliton structures of an algebraic nature, which are simpler to handle.

\begin{theorem}\label{i-rsequiv}
For a simply connected homogeneous space $(G/K,\gamma)$ endowed with a $G$-invariant geometric structure $\gamma$ of type $(r,s)$, the following conditions are equivalent:
\begin{itemize}
\item[(i)] The bracket flow solution starting at $\lb$ is given by
$$
\mu(t)=c(t)\cdot\lb, \qquad\mbox{for some}\quad c(t)>0, \quad c(0)=1.
$$
\item[(ii)] The operator $Q(\gamma)\in\qg_\gamma$ such that $\theta(Q(\gamma))=q(\gamma)$ satisfies
\begin{equation}\label{i-as}
Q(\gamma)=cI+D_\pg, \qquad \mbox{for some} \quad c\in\RR, \quad D=\left[\begin{smallmatrix} 0&0\\ 0&D_\pg \end{smallmatrix}\right] \in\Der(\ggo).
\end{equation}
\end{itemize}
In that case, $(G/K,\gamma)$ is a soliton geometric structure with
$$
q(\gamma)=(s-r)c\gamma-\lca_{X_D}\gamma,
$$
where $X_D$ denotes the vector field on $G/K$ defined by the one-parameter subgroup of $\Aut(G)$ attached to the derivation $D$.
\end{theorem}

A homogeneous space $(G/K,\gamma)$ endowed with a $G$-invariant geometric structure $\gamma$ and a reductive decomposition $\ggo=\kg\oplus\pg$ is said to be an {\it algebraic soliton} if condition \eqref{i-as} holds.  The concept of algebraic soliton has been very fruitful in the study of homogeneous Ricci solitons since its introduction in \cite{soliton} (see also \cite{SCF,Frn2,SCFmuA,CRF} for the symplectic curvature flow and Chern-Ricci flow cases).  It is a useful tool to address the existence problem for soliton structures, as well as to study their uniqueness, structure and low-dimensional classification.

We use this approach to exhibit in Section \ref{G2sol} an expanding soliton closed $G_2$-structure on a nilpotent Lie group for the Laplacian flow introduced by R. Bryant in \cite{Bry}.  As far as we know, this is the first example known of a Laplacian soliton which is not an eigenvector (see \cite{LtyWei} and the references therein).

\section{Some linear algebra related to geometric structures}\label{LAgs}

Let $\gamma$ be a geometric structure on a differentiable manifold $M$, e.g.\ a Riemannian metric, a complex structure, a symplectic structure, an almost-hermitian structure, a $G_2$-structure, etc.  We also denote by $\gamma$ the corresponding tensor field, or the set of tensor fields, defining the geometric structure, i.e.\ a symmetric $2$-tensor $g$, a $(1,1)$-tensor $J$, a $2$-form $\omega$, a tern $\gamma=(\omega,g,J)$ such that $\omega=g(J\cdot,\cdot)$, a $3$-form $\vp$, etc., respectively.

After fixing a point $p\in M$ and a basis $\{ e_1,\dots,e_n\}$ of the tangent space $T_pM$, we obtain a tensor $\gamma=\gamma_p$ on the vector space $\RR^n=T_pM$. As one can observe in the examples above, the tensor $\gamma$ is always non-degenerate in some sense: $g$ is positive definite, $J^2=-I$, $\omega$ is non-degenerate and $\vp$ is positive.  Moreover, a common property that is satisfied by all these geometric structures is that the orbit
\begin{equation}\label{nondeg3}
\Gl_n(\RR)\cdot\gamma
\end{equation}
is open in the vector space $T$ of all tensors of the same type as $\gamma$ and consists precisely of those tensors which are non-degenerate.  This property will be assumed to hold for $\gamma$ in this paper.  In particular,
\begin{equation}\label{nondeg}
\theta(\glg_n(\RR))\gamma = T,
\end{equation}
where $\theta:\glg_n(\RR)\longrightarrow \End(T)$ is the representation obtained as the derivative of the natural left $\Gl_n(\RR)$-action on tensors $(h,\gamma)\mapsto h\cdot\gamma=(h^{-1})^*\gamma$ (i.e.\ $\theta(A)\gamma=\ddt|_0e^{tA}\cdot\gamma$).  The Lie algebra of the stabilizer subgroup
$$
G_\gamma:=\{ h\in\Gl_n(\RR):h\cdot\gamma=\gamma\},
$$
is given by
$$
\ggo_\gamma:=\{ A\in\glg_n(\RR):\theta(A)\gamma=0\}.
$$
It follows from \eqref{nondeg3} that the set of all nondegenerate tensors of the same type as $\gamma$ is parameterized by the homogeneous space $\Gl_n(\RR)/G_\gamma$.  We consider an $\Ad(G_\gamma)$-invariant subspace $\qg_\gamma\subset\glg_n(\RR)$ such that
$$
\glg_n(\RR)=\ggo_\gamma\oplus\qg_\gamma.
$$
Since $\ggo_{h\cdot\gamma}=h\ggo_\gamma h^{-1}$, we can set for simplicity $\qg_{h\cdot\gamma}:=h\qg_\gamma h^{-1}$ for all $h\in\Gl_n(\RR)$.  By \eqref{nondeg} we have that
$$
\theta(\qg_\gamma)\gamma=T;
$$
moreover, for every tensor $q\in T$, there exists a unique operator $Q\in\qg_\gamma$ such that
\begin{equation}\label{nondeg2}
q=\theta(Q)\gamma.
\end{equation}

By varying the point $p$, what one obtains is a tensor subbundle $T\subset T^{r,s}M$, subbundles $\ggo_\gamma$ and $\qg_\gamma$ of $\End(TM)=T^{1,1}M$ such that $\End(TM)=\ggo_\gamma\oplus\qg_\gamma$ and  a linear map $j_\gamma:T\longrightarrow\qg_\gamma$ defined locally by $\theta(j_\gamma(q))\gamma=q$ (i.e. $j_\gamma(q)=Q$ if \eqref{nondeg2} holds).  Note that $j_\gamma$ is an isomorphism with inverse $i_\gamma:\qg_\gamma\longrightarrow T$ given by $i_\gamma(Q):=\theta(Q)\gamma$.  It holds that
\begin{equation}\label{j-equiv}
j_{h\cdot\gamma}(q)=hj_\gamma(h^{-1}\cdot q)h^{-1}, \qquad\forall q\in T, \quad h\in\Gl_n(\RR).
\end{equation}

Let $q=q(\delta)$ be a smooth tensor field on $M$ associated to each geometric structure $\delta$, of the same type as $\gamma$ (i.e. $q(\delta), \delta\in T$).  It follows from \eqref{nondeg2} that for each $\gamma$, there is a unique smooth $(1,1)$-tensor field $Q(\gamma)=j_\gamma(q(\gamma))$ on $M$, a section of the subbundle $\qg_\gamma$, such that at each point,
\begin{equation}\label{nondeg4}
q(\gamma)=\theta(Q(\gamma))\gamma, \qquad Q(\gamma)\in\qg_\gamma.
\end{equation}
Assume now that the map $\delta\mapsto q(\delta)$ is diffeomorphism equivariant: $q(\vp^*\delta)=\vp^*q(\delta)$ for any $\vp\in\Diff(M)$ (e.g.\ any curvature tensor, or Laplacian associated to a geometric structure, or the gradient field of any natural geometric functional).  Then,
\begin{equation}\label{Qequiv}
Q({\vp^*\gamma}) = (d\vp)^{-1}Q(\gamma) d\vp,
\end{equation}
or equivalently at each point, $Q(h\cdot\gamma)=hQ(\gamma)h^{-1}$ for any $h\in\Gl_n(\RR)$.  Indeed, for $h:=d\vp|_p$ we have that
\begin{align*}
\theta(Q({\vp^*\gamma}))h^{-1}\cdot\gamma & = \theta(Q({\vp^*\gamma}))\vp^*\gamma = q(\vp^*\gamma)=\vp^*q(\gamma) \\
& = h^{-1}\cdot\theta(Q(\gamma))\gamma = \theta(h^{-1}Q(\gamma)h)h^{-1}\cdot\gamma,
\end{align*}
and so $Q(\vp^*\gamma)-h^{-1}Q(\gamma)h\in\ggo_{h^{-1}\cdot\gamma}\cap\qg_{h^{-1}\cdot\gamma} = 0$.

\begin{example}\label{exa}
We now review each of the particular geometric structures mentioned at the beginning of the section.
\begin{itemize}
\item[(i)] For a Riemannian metric $\gamma=g$, by assuming that the fixed basis $\{ e_i\}$ of $\RR^n=T_pM$ is orthonormal, we obtain that
$$
\begin{array}{c}
T=\sca^2(\RR^n)^*, \qquad h\cdot g=g(h^{-1}\cdot,h^{-1}\cdot), \qquad \theta(A)g=-g(A\cdot,\cdot)-g(\cdot,A\cdot), \\ \\
G_\gamma=\Or(n), \qquad \ggo_\gamma=\sog(n), \qquad \qg_\gamma=\sym(n):=\{ A\in\glg_n(\RR):A^t=A\}.
\end{array}
$$
Thus what condition \eqref{nondeg2} is asserting is simply that for any symmetric $2$-tensor $q\in\sca^2(\RR^n)^*$, there exists a unique $Q\in\sym(n)$ such that $q=g(-2Q\cdot,\cdot)$.  For example, if $q(g)=\ricci(g)$, the Ricci tensor of $g$, then $Q(g)=-\unm\Ricci(g)$, where $\Ricci(g)$ is the Ricci operator of $g$.

\item[(ii)]  If $\gamma=J$ is a complex (or almost-complex) structure and $\dim{M}=2n$, then
$$
\begin{array}{c}
T=[\glg_{2n}(\RR),J], \qquad h\cdot J=hJh^{-1}, \qquad \theta(A)J=[A,J], \\ \\
G_\gamma=\Gl_n(\CC), \qquad \ggo_\gamma=\glg_n(\CC), \qquad \qg_\gamma=\{ A\in\glg_{2n}(\RR):AJ=-JA\}.
\end{array}
$$

\item[(iii)] For a symplectic structure $\gamma=\omega$, $\dim{M}=2n$, we have that
$$
\begin{array}{c}
T=\Lambda^2(\RR^{2n})^*, \qquad h\cdot\omega=\omega(h^{-1}\cdot,h^{-1}\cdot), \qquad \theta(A)\omega=-\omega(A\cdot,\cdot)-\omega(\cdot,A\cdot), \\ \\
G_\gamma=\Spe(n,\RR), \qquad \ggo_\gamma=\spg(n,\RR), \qquad \qg_\gamma=\{ A\in\glg_{2n}(\RR):\omega(A\cdot,\cdot)=\omega(\cdot,A\cdot)\}.
\end{array}
$$

\item[(iv)] An almost-hermitian (or hermitian, or almost-K\"ahler) structure $\gamma=(\omega,g,J)$, $\omega=g(J\cdot,\cdot)$, $\dim{M}=2n$, gives
$$
\begin{array}{c}
T=\left\{ (p,q,R)\in\Lambda^2(\RR^{2n})^*\oplus\sca^2(\RR^{2n})^*\oplus [\glg_{2n}(\RR),J] : p=q(J\cdot,\cdot)+g(R\cdot,\cdot)\right\}, \\ \\ h\cdot\gamma=(h\cdot\omega,h\cdot g,h\cdot J), \qquad \theta(A)\gamma=(\theta(A)\omega,\theta(A)g,\theta(A)J), \\ \\
G_\gamma=\U(n), \qquad \ggo_\gamma=\ug(n), \qquad \qg_\gamma=\qg_1\oplus\qg_2\oplus\qg_3,
\end{array}
$$
where,
$$
\begin{array}{l}
\qg_1:=\{ A\in\glg_{2n}(\RR):A^t=A, \quad AJ=-JA\}, \\ \\
\qg_2:=\{ A\in\glg_{2n}(\RR):A^t=-A, \quad AJ=-JA\}, \\ \\
\qg_3=\herm(n):=\{ A\in\glg_{2n}(\RR):A^t=A, \quad AJ=JA\}.
\end{array}
$$
We note that the following decompositions also hold:
$$
\spg(n,\RR)=\ug(n)\oplus\qg_1, \qquad \sog(2n)=\ug(n)\oplus\qg_2, \qquad \glg_n(\CC)=\ug(n)\oplus\qg_3,
$$
Note that $\dim{\qg_1}=n^2+n$, $\dim{\qg_2}=n^2-n$, $\dim{\qg_3}=n^2$ and so $\dim{T}=\dim{\qg_\gamma}=3n^2$.  According to \eqref{nondeg2}, for each $(p,q,R)\in T$ there exists a unique operator $Q=Q_1+Q_2+Q_3\in\qg_\gamma$, $Q_i\in\qg_i$, such that
$$
\begin{array}{c}
p=\theta(Q)\omega=\theta(Q_2+Q_3)\omega, \qquad q=\theta(Q)g=\theta(Q_1+Q_3)g, \\ \\  R=\theta(Q)J=\theta(Q_1+Q_2)J.
\end{array}
$$
\item[(v)] A $G_2$-structure on a $7$-dimensional differentiable manifold $M$ is a $3$-form $\vp$ which can be written on each tangent space as
$$
\vp=e^{127}+e^{347}+e^{567}+e^{135}-e^{146}-e^{236}-e^{245},
$$
with respect to some basis $\{ e_1,\dots,e_7\}$ of $\RR^7\equiv T_pM$.  Here $e^{ijk}:=e^i\wedge e^j\wedge e^k$ and $\{ e^i\}$ is the dual basis of $\{ e_i\}$.  By varying the basis, one obtains the open orbit $\Gl_7(\RR)\cdot\vp\subset\Lambda^3(\RR^7)^*$ of {\it positive} $3$-forms.  Thus for a $G_2$-structure $\gamma=\vp$, $n=7$ and
$$
\begin{array}{c}
T=\Lambda^3(\RR^7)^*, \qquad h\cdot\vp=\vp(h^{-1}\cdot,h^{-1}\cdot,h^{-1}\cdot), \\ \\ \theta(A)\vp=-\vp(A\cdot,\cdot,\cdot)-\vp(\cdot,A\cdot,\cdot)-\vp(\cdot,\cdot,A\cdot), \\ \\
G_\gamma=G_2\simeq\Aut(\OO), \qquad \ggo_\gamma=\ggo_2, \qquad \qg_\gamma=\qg_1\oplus\qg_7\oplus\qg_{27},
\end{array}
$$
where $\qg_1=\RR I$ is the one-dimensional trivial representation of $\ggo_2$, $\qg_7$ is the ($7$-dimensional) standard representation and $\qg_{27}$ is the other fundamental representation of $\ggo_2$, which has dimension $27$ (see \cite{Krg}).  It follows that
$$
\sog(7)=\ggo_2\oplus\qg_7, \qquad \sym(7)=\qg_1\oplus\qg_{27}, \qquad \qg_{27}=\sym_0(7),
$$
where $\sym_0(7):=\{ A\in\sym(7):\tr{A}=0\}$.

\item[(vi)] One may also fix a complex manifold $(M,J)$ rather than a differentiable manifold and consider a hermitian metric $\gamma=g$ on $(M,J)$.  In this case, $\dim{M}=2n$, $\Gl_n(\CC)$ plays the role of $\Gl_n(\RR)$ in the previous cases, in particular in property \eqref{Qequiv}, and we have that
$$
\begin{array}{c}
T=\{ g\in\sca^2(\RR^{2n})^*:g(J\cdot,J\cdot)=g\}, \qquad G_\gamma=\U(n), \\ \\
\glg_n(\CC)=\ggo_\gamma\oplus\qg_\gamma,\qquad \ggo_\gamma=\ug(n), \qquad \qg_\gamma=\herm(n).
\end{array}
$$
The assignment $\delta\mapsto q(\delta)$ above, which will be used to define a geometric flow, is therefore assumed to be invariant by bi-holomorphic maps of $(M,J)$ rather than by diffeomorphisms of $M$ (see Section \ref{CRF-sec} for an application in the complex case).  

\item[(vii)] The symplectic analogous to part (vi) consists in fixing a symplectic manifold $(M,\omega)$ and consider a compatible metric $\gamma=g$, giving rise to an almost-K\"ahler manifold $(M,\omega,g)$.  If $\dim{M}=2n$, then the group involved is now $\Spe(n,\RR)$ and
$$
\begin{array}{c}
T=\{ g\in\sca^2(\RR^{2n})^*:\omega=g(J\cdot,\cdot)\;\mbox{implies}\; J^2=-I\}, \qquad G_\gamma=\U(n), \\ \\
\spg(n,\RR)=\ggo_\gamma\oplus\qg_\gamma,\qquad  \ggo_\gamma=\ug(n), \qquad \qg_\gamma=\qg_1,
\end{array}
$$
where $\qg_1$ is as in part (iv).  The assignment $\delta\mapsto q(\delta)$ is here assumed to be invariant by symplectomorphisms of $(M,\omega)$.
\end{itemize}
\end{example}

\section{The space of homogeneous spaces}\label{hm}

Our aim in this section is to describe a framework developed in \cite{spacehm} in the Riemannian case, which allows us to work on the `space of homogeneous manifolds', by parameterizing  the set of all homogeneous spaces of dimension $n$ and isotropy dimension $q$ by a subset $\hca_{q,n}$ of the variety of $(q+n)$-dimensional Lie algebras.

A connected differentiable manifold $M$ is called {\it homogeneous} if there is a Lie group $G$ acting smoothly and transitively on M.  Each transitive Lie group $G$ (which can be assumed to be connected) gives rise to a presentation of $M$ as a {\it homogeneous space} $G/K$, where $K$ is the isotropy subgroup of $G$ at some point $o\in M$.  Conversely, any closed subgroup $K\subset G$ of a Lie group $G$ defines a homogeneous manifold $M=G/K$ with isotropy subgroup at the origin $o:=eK\in G/K$ given by $G_o=K$.

Consider an $\Ad(K)$-invariant direct sum $\ggo=\kg\oplus\pg$, where $\ggo$ and $\kg$ are respectively the Lie algebras of $G$ and $K$.  This is called a {\it reductive decomposition}, its existence is for instance guaranteed by the (relative) compactness of $\Ad(K)$ (e.g. if $G/K$ admits a $G$-invariant Riemannian metric) and is in general non-unique.  Note that $\pg$ can be naturally identified with the tangent space
$$
\pg\equiv T_oM=T_oG/K,
$$
via the derivative $d\pi|_e:\ggo\longrightarrow T_oG/K$ of the natural projection $\pi:G\longrightarrow G/K$.  This is equivalent to take the value at the origin $o$ of the Killing vector fields corresponding to elements of $\pg$ (i.e.\ $X_o=\ddt|_0\exp{tX}(o)$ for any $X\in\pg$).

Any homogeneous space $G/K$ will be assumed to be {\it almost-effective} for simplicity, i.e.\ the normal subgroup $\{ g\in G:ghK=hK, \;\forall h\in G\}\subset K$ is discrete.

\subsection{Varying Lie brackets viewpoint}\label{varhs}
Let us fix a $(q+n)$-dimensional real vector space $\ggo$ together with a direct sum decomposition
\begin{equation}\label{fixdec}
\ggo=\kg\oplus\pg, \qquad \dim{\kg}=q, \qquad \dim{\pg=n}.
\end{equation}
The space of all skew-symmetric algebras (or brackets) of dimension $q+n$ is parameterized by the vector space
$$
\Lambda^2\ggo^*\otimes\ggo=\{\mu:\ggo\times\ggo\longrightarrow\ggo : \mu\; \mbox{bilinear and
skew-symmetric}\}.
$$

\begin{definition}\label{hqn} We consider the subset $\hca_{q,n}\subset\Lambda^2\ggo^*\otimes\ggo$, or more precisely $\hca_{\ggo=\kg\oplus\pg}$ if preferred, of those brackets $\mu$ such that:
\begin{itemize}
\item [{\bf (h1)}]  $\mu$ satisfies the Jacobi condition, $\mu(\kg,\kg)\subset\kg$ and $\mu(\kg,\pg)\subset\pg$.

\item[{\bf (h2)}] If $G_\mu$ denotes the simply connected Lie group with Lie algebra $(\ggo,\mu)$ and $K_\mu$ is the connected Lie subgroup of $G_\mu$ with Lie algebra $\kg$, then $K_\mu$ is closed in $G_\mu$.

\item[{\bf (h3)}] $\{ Z\in\kg:\mu(Z,\pg)=0\}=0$.
\end{itemize}
\end{definition}

It follows from (h1) and (h2) that each $\mu\in\hca_{q,n}$ defines a unique simply connected homogeneous space,
\begin{equation}\label{hsmu}
\mu\in\hca_{q,n}\rightsquigarrow G_{\mu}/K_{\mu},
\end{equation}
with reductive decomposition $\ggo=\kg\oplus\pg$.  It is almost-effective by (h3).  We note that any $n$-dimensional simply connected homogeneous space $G/K$ which is almost-effective and admits a reductive decomposition can be identified with some $\mu\in\hca_{q,n}$, where $q=\dim{K}$.  Indeed, $G$ can be assumed to be simply connected without losing almost-effectiveness, and we can identify any reductive decomposition with $\ggo=\kg\oplus\pg$ at a vector space level. In this way, $\mu$ will be precisely the Lie bracket of $\ggo$.

There is a natural linear action of $\Gl(\ggo)$ on $\Lambda^2\ggo^*\otimes\ggo$ given by
\begin{equation}\label{action}
(h\cdot\mu)(X,Y)=h\mu(h^{-1}X,h^{-1}Y), \qquad X,Y\in\ggo.
\end{equation}
Note that $h$ is precisely an isomorphism between the Lie algebras $(\ggo,\mu)$ and $(\ggo,h\cdot\mu)$.  If $\mu\in\hca_{q,n}$, then $h\cdot\mu\in\hca_{q,n}$ for any $h\in\Gl(\ggo)$ leaving $\kg$ and $\pg$ invariant (i.e.\ $h\in\Gl(\kg)\times\Gl(\pg)$).  In that case, the isomorphism $\widetilde{\vp}:G_\mu\longrightarrow G_{h\cdot\mu}$ with derivative $d\widetilde{\vp}|_e=h$ satisfies that $\widetilde{\vp}(K_{\mu})=K_{h\cdot\mu}$ and so it defines an equivariant diffeomorphism
\begin{equation}\label{defphi}
\vp:G_\mu/K_\mu\longrightarrow G_{h\cdot\mu}/K_{h\cdot\mu}, \qquad \vp(aK_\mu):=\widetilde{\vp}(a)K_{h\cdot\mu}.
\end{equation}

Concerning the question of what kind of subset of $\Lambda^2\ggo^*\otimes\ggo$ the space $\hca_{q,n}$ is, we note that condition (h1) is closed as it is defined by polynomial equations on $\mu$.  On the contrary, (h3) is open and (h2) may impose very subtle conditions on $\mu$ (see \cite[Examples 3.4, 3.6]{spacehm}).  Note that $\hca_{q,n}$ is a cone, i.e. invariant by nonzero scaling.

\begin{example}\label{ex0-n}
If $q=0$, then conditions (h2) and (h3) trivially hold and (h1) is just the Jacobi condition for $\mu$.  Thus $\hca_{0,n}=\lca_n$, the variety of $n$-dimensional Lie algebras, and the set $\{G_\mu:\mu\in\lca_n\}$ parameterizes the set of all simply connected Lie groups of dimension $n$.
\end{example}

\subsection{Invariant geometric structures}\label{invgs}
Any $G$-invariant geometric structure on a homogeneous space $M=G/K$ with reductive decomposition $\ggo=\kg\oplus\pg$ is determined by a tensor, or a set of tensors, $\gamma$ on $\pg\equiv T_oM$ which is $\Ad(K)$-invariant.  This means that $(\Ad(k)|_\pg)\cdot\gamma=\gamma$ for any $k\in K$, or equivalently if $K$ is connected, $\theta(\ad{Z}|_\pg)\gamma=0$ for all $Z\in\kg$ (see Section \ref{LAgs} for the notation).  If $G/K$ is not reductive, then one can work with the identification $\ggo/\kg\equiv T_oM$, though there will always be a reductive decomposition as soon as a Riemannian metric is involved in some way in the geometric structure $\gamma$.   

If in addition to the connectedness of $K$ we assume that $G$ is simply connected (in particular, $G/K$ is simply connected), then we can apply \eqref{Qequiv} to the equivariant diffeomorphism $\vp$ defined by the automorphism $\Ad(k)\in\Aut(\ggo)\equiv\Aut(G)$.  Since $\vp^*\gamma=\gamma$, we obtain that the operator $Q(\gamma)\in\End(\pg)$ attached to any $G$-invariant tensor field $q(\gamma)$ satisfies $[\Ad(k)|_\pg,Q(\gamma)]=0$ for each $k\in K$, or equivalently,
\begin{equation}\label{adZQ}
[\ad{Z}|_\pg,Q(\gamma)]=0, \qquad\forall Z\in\kg.
\end{equation}

In order to study invariant geometric structures on homogeneous manifolds, we can therefore fix a tensor $\gamma$ on $\pg$ in addition to the vector space decomposition fixed in \eqref{fixdec} and consider the following extra condition:

\begin{itemize}
\item[{\bf (h4)}] $\gamma$ is $\ad_{\mu}{\kg}$-invariant (i.e.\ $\theta(\ad_{\mu}{Z}|_{\pg})\gamma=0$ for all $Z\in\kg$, or equivalently, $\ad_\mu{\kg}|_\pg\subset\ggo_\gamma$).
\end{itemize}

If we consider the subset of $\hca_{q,n}$ given by
\begin{equation}\label{hqng}
\hca_{q,n}(\gamma):=\{\mu\in\hca_{q,n}: \mbox{condition (h4) holds true for}\,\mu\},
\end{equation}
then each $\mu\in\hca_{q,n}(\gamma)$ now defines a unique (almost-effective and simply connected) homogeneous space endowed with an invariant geometric structure,
\begin{equation}\label{hsmu-gamma}
\mu\in\hca_{q,n}(\gamma)\rightsquigarrow\left(G_{\mu}/K_{\mu},\gamma_\mu\right),
\end{equation}
with reductive decomposition $\ggo=\kg\oplus\pg$ and
$$
\gamma_\mu(o)=\gamma, \qquad\forall\mu\in\hca_{q,n}(\gamma).
$$
Indeed, it follows from (h4) that $\gamma$ is $\Ad(K_{\mu})$-invariant as $K_{\mu}$ is connected and thus the geometric structure $\gamma_\mu$ is $G_\mu$-invariant.  $\left(G_{\mu}/K_{\mu},\gamma_\mu\right)$ will be sometimes denoted by $\left(G_{\mu}/K_{\mu},\gamma\right)$, as the whole `linear algebra' part of $\gamma_\mu$ has been fixed.

\begin{remark}\label{rem4}
The complex and symplectic cases mentioned in Remarks \ref{rem1}, \ref{rem2} and \ref{rem3} can be treated in this section as special cases of almost-hermitian structures, by adding to the definition of $\hca_{q,n}(\gamma)$ the condition of integrability of $(G_\mu/K_\mu,J)$,
\begin{itemize}
\item[{\bf (h5-$J$)}] $\mu_\pg(JX,JY)=\mu_\pg(X,Y)+J\mu_\pg(JX,Y)+J\mu_\pg(X,JY)$, for all $X,Y\in\pg$,
\end{itemize}
in the complex case, and the condition $d_\mu\omega=0$, where $d_\mu$ is the differential of forms on the manifold $G_\mu/K_\mu$, in the symplectic case, which is equivalent to
\begin{itemize}
\item[{\bf (h5-$\omega$)}] $\omega(\mu_\pg(X,Y),Z)+\omega(\mu_\pg(Y,Z),X)+\omega(\mu_\pg(Z,X),Y)=0$, for all $X,Y,Z\in\pg$.
\end{itemize}
\end{remark}

The group $\Gl(\kg)\times G_\gamma$, where $G_\gamma\subset\Gl(\pg)$ is the stabilizer subgroup of $\gamma$ (with Lie algebra $\ggo_\gamma$), leaves the set $\hca_{q,n}(\gamma)$ invariant and for any  $h=(h_\kg,h_\pg)\in\Gl(\kg)\times G_\gamma$, $\mu\in\hca_{q,n}(\gamma)$, one obtains that the equivariant diffeomorphism
\begin{equation}\label{phi-equiv}
\vp:\left(G_{\mu}/K_{\mu},\gamma_\mu\right)\longrightarrow \left(G_{h\cdot\mu}/K_{h\cdot\mu},\gamma_{h\cdot\mu}\right),
\end{equation}
defined in \eqref{defphi} is an {\it equivalence of geometric structures}, that is, $\gamma_\mu=\vp^*\gamma_{h\cdot\mu}$.  This follows from the invariance of the structures and the fact that $\vp^*\gamma=\gamma$ at the origins (recall that $d\vp|_o=\hg_\pg\in G_\gamma$).  In that case, any tensor field $q$ defined for this kind of geometric structures will satisfy that
\begin{equation}\label{Qmuh}
  q(\gamma_{h\cdot\mu})=h_\pg\cdot q(\gamma_\mu), \qquad Q_{h\cdot\mu}=h_\pg Q_\mu h_\pg^{-1}, \qquad\forall h=(h_\kg,h_\pg)\in\Gl(\kg)\times G_\gamma,
\end{equation}
where $Q_\mu\in\qg_{\gamma}\subset\End(\pg)$ is the operator defined by
\begin{equation}\label{defQmu}
\theta(Q_\mu)\gamma=q(G_\mu/K_\mu,\gamma_\mu).
\end{equation}
The statement on $Q$ can be proved in much the same way as \eqref{Qequiv} by using that $h_\pg\cdot\gamma=\gamma$ as follows:
$$
\theta(Q_{h\cdot\mu})\gamma = q(\gamma_{h\cdot\mu}) = h_\pg\cdot q(\gamma_\mu) = h_\pg\cdot\theta(Q_\mu)\gamma = \theta(h_\pg Q_\mu h_\pg^{-1})h_\pg\cdot\gamma = \theta(h_\pg Q_\mu h_\pg^{-1})\gamma,
$$
and so $Q_{h\cdot\mu}-h_\pg Q_\mu h_\pg^{-1}\in\ggo_{\gamma}\cap\qg_{\gamma}=0$.

Recall from \eqref{phi-equiv} that, geometrically, the orbit $(\Gl(\kg)\times G_\gamma)\cdot\mu\subset\hca_{q,n}(\gamma)$ does not bring anything new.  The following result gives a useful geometric meaning to the action of a subset of $\Gl(\ggo)$ on $\hca_{q,n}(\gamma)$, which is much larger than the subgroup $\Gl(\kg)\times G_\gamma$.

\begin{proposition}\label{const}
Given $\mu\in\hca_{q,n}(\gamma)$ and $h=(h_\kg,h_\pg)\in\Gl(\kg)\times\Gl(\pg)$, it holds that $h\cdot\mu\in\hca_{q,n}(\gamma)$ if and only if
\begin{equation}\label{adkh}
h_\pg\ad_{\mu}{\kg}|_{\pg}h_\pg^{-1}\subset\ggo_\gamma.
\end{equation}
In that case, $\left(G_{h\cdot\mu}/K_{h\cdot\mu},\gamma_{h\cdot\mu}\right)$ is equivalent to $\left(G_{\mu}/K_{\mu}, h_\pg^*\cdot\gamma\right)$.
\end{proposition}

\begin{proof}
We have that $h\cdot\mu\in\hca_{q,n}(\gamma)$ if and only if $\ad_{h\cdot\mu}{\kg}|_\pg=h_\pg\ad_\mu{\kg}|_\pg h_\pg^{-1}\subset\ggo_\gamma$.  The equivalence is provided by the equivariant diffeomorphism defined in \eqref{defphi} (recall that $\hg_\pg^*\gamma=h_\pg^{-1}\cdot\gamma$).
\end{proof}

\begin{corollary}\label{const-cor}
The subset
$$
\big\{ h\cdot\mu:h=(h_\kg,h_\pg)\in\Gl(\kg)\times\Gl(\pg), \quad h_\pg\;\mbox{satisfies condition}\; \eqref{adkh} \big\}\subset\hca_{q,n}(\gamma),
$$
parameterizes the set of all $G_\mu$-invariant geometric structures of the same class as $\gamma$ on the homogeneous space $G_\mu/K_\mu$ up to equivariant equivalence.
\end{corollary}

Assume that $\gamma$ is an $(r,s)$-tensor field, i.e.\ $r$-times covariant and $s$-times contravariant (e.g.\ a Riemannian metric is a $(2,0)$-tensor field).  It follows that
$$
(cI)\cdot\gamma=c^{s-r}\gamma, \qquad \theta(I)\gamma=(s-r)\gamma.
$$
By setting $h:=(I,\frac{1}{c}I)\in\Gl(\kg)\times\Gl(\pg)$, $c\ne 0$, we obtain from Proposition \ref{const} that the rescaled $G_\mu$-invariant geometric structure $c^{s-r}\gamma_\mu$ on $G_\mu/K_\mu$ is equivalent to the element of $\hca_{q,n}(\gamma)$ defined by $c\cdot\mu:=(I,\tfrac{1}{c}I)\cdot\mu$, that is,
\begin{equation}\label{scmu}
c\cdot\mu|_{\kg\times\kg}=\mu, \qquad c\cdot\mu|_{\kg\times\pg}=\mu, \qquad c\cdot\mu|_{\pg\times\pg}=c^2\mu_{\kg}+c\mu_{\pg},
\end{equation}
where the subscripts denote the $\kg$- and $\pg$-components of $\mu|_{\pg\times\pg}$ given by
\begin{equation}\label{decmu}
\mu(X,Y)=\mu_{\kg}(X,Y)+\mu_{\pg}(X,Y), \qquad \mu_{\kg}(X,Y)\in\kg, \quad \mu_{\pg}(X,Y)\in\pg, \quad\forall X,Y\in\pg.
\end{equation}
The $\RR^*$-action on $\hca_{q,n}(\gamma)$, $\mu\mapsto c\cdot\mu$, can therefore be considered as a {\it geometric scaling}: 
\begin{quote}
$(G_{c\cdot\mu}/K_{c\cdot\mu},\gamma)$ is equivalent to $(G_{\mu}/K_{\mu},c^{s-r}\gamma)$.
\end{quote}

\begin{remark}\label{rem5}
The above described geometric scaling is not allowed in the symplectic case, i.e. if one has fixed a nondegenerate $2$-form $\omega$, since the map $\frac{1}{c}I$ does not belong to $\Spe(n,\RR)$.
\end{remark}

Given any $\mu\in\hca_{q,n}$, if we define $\lambda\in\hca_{q,n}$ by $\lambda|_{\kg\times\ggo}:=\mu$, $\lambda|_{\pg\times\pg}:=0$ (note that conditions (h1)-(h3) also hold for $\lambda$), what we obtain is the Euclidean space
$$
G_{\lambda}/K_{\lambda}=(K\ltimes\RR^n)/K=\RR^n,
$$
for some closed subgroup $K\subset\Gl_n(\RR)$.  Note that $\lambda=\lim\limits_{c\to 0}c\cdot\mu$, so this can be viewed as $\left(G_{c\cdot\mu}/K_{c\cdot\mu},\gamma\right)$ converging in the pointed sense to the (flat) geometric structure $(\RR^n,\gamma)$, as $c\to 0$ (see Section \ref{converg}).

\subsection{Degenerations and pinching conditions}\label{deg-pin}
An elementary but crucial observation is that any kind of geometric quantity associated to the manifold $\left(G_\mu/K_\mu,\gamma\right)$ depends continuously on the Lie bracket $\mu\in\hca_{q,n}(\gamma)\subset\Lambda^2\ggo^*\otimes\ggo$.  This can be used to study pinching curvature properties as follows.

\begin{definition}\label{degen-def}
Given $\mu,\lambda\in\hca_{q,n}(\gamma)$, we say that $\mu$ {\it degenerates to} $\lambda$, denoted by $\mu\rightarrow\lambda$, if $\lambda$ belongs to the closure of the subset
$$
\big\{ h\cdot\mu:h=(h_\kg,h_\pg)\in\Gl(\kg)\times\Gl(\pg), \quad h_\pg\;\mbox{satisfies condition}\; \eqref{adkh} \big\}\subset\hca_{q,n}(\gamma),
$$
relative to the usual vector space topology of $\Lambda^2\ggo^*\otimes\ggo$.
\end{definition}

The geometric meaning of the above subset of $\hca_{q,n}(\gamma)$ has been explained in Corollary \ref{const-cor}.

\begin{proposition}\label{deg-pin}
If $\mu\rightarrow\lambda$, $\mu,\lambda\in\hca_{q,n}(\gamma)$, and $G_\lambda/K_\lambda$ admits a $G_\lambda$-invariant geometric structure satisfying a strict pinching curvature condition, then there is also a $G_\mu$-invariant structure on $G_\mu/K_\mu$ for which the same pinching curvature condition holds.
\end{proposition}

\begin{proof}
If $\mu\rightarrow\lambda$, then there is a sequence $h_k\in\Gl(\kg)\times\Gl(\pg)$ with $(h_k)_\pg$ satisfying \eqref{adkh} such that $h_k\cdot\mu\to\lambda$, as $k\to\infty$.  The pinching curvature condition will therefore hold for $\left( G_{h_k\cdot\mu}/K_{h_k\cdot\mu},\gamma\right)$ for sufficiently large $k$, which implies that it holds for the corresponding equivalent (see Proposition \ref{const}) invariant structure $(h_k)_\pg^*\gamma$ on $G_\mu/K_\mu$ for sufficiently large $k$, concluding the proof.
\end{proof}

\subsection{Convergence}\label{converg}
We follow the lines of the article \cite{spacehm} about convergence of homogeneous manifolds. In order to derive natural notions of convergence, of a sequence $(M_k=G_k/K_k,\gamma_k)$ of homogeneous spaces endowed with a geometric structure, to a homogeneous space $(M=G/K,\gamma)$, we start by requiring the existence of a sequence $\Omega_k\subset M$ of open neighborhoods of the origin $o\in M$ together with embeddings $\phi_k:\Omega_k\longrightarrow M_k$ such that $\phi_k(o)=o$ and $\phi_k^*\gamma_k\to \gamma$ {\it smoothly} as $k\to\infty$.  This means that the tensor field $\phi_k^*\gamma_k-\gamma$ and its covariant derivatives of all orders with respect to some fixed Riemannian metric on $M$ each converge uniformly to zero on compact subsets of $M$ eventually contained in all $\Omega_k$.  Equivalently, for a chart contained in all $\Omega_k$ for sufficiently large $k$, the partial derivative $\partial^\alpha(\gamma_k)_{J}$ of the coordinates $(\gamma_k)_{J}$ of the tensor fields converge to $\partial^\alpha \gamma_{J}$ uniformly, as $k\to\infty$, for any multi-indices $\alpha$ and $J$.
According to the different conditions one may require on the size of the $\Omega_k$'s, we have the following notions of convergence in increasing degree of strength:

\begin{itemize}
\item {\it infinitesimal}: no condition on $\Omega_k$, it may even happen that $\bigcap\Omega_k=\{ o\}$.  So $\phi_k^*\gamma_k\to\gamma$ smoothly as $k\to\infty$ {\it at} $p$, in the sense that for any $\epsilon>0$, there exists $k_0=k_0(\epsilon)$ such that for $k\geq k_0$,
$$
\sup_{\Omega_k} |\nabla^j(\phi_k^*\gamma_k-\gamma)|<\epsilon, \qquad\forall j\in\ZZ_{\geq 0}.
$$
The infinitesimal convergence of homogeneous manifolds is somewhat weak, notice that actually only the germs of the geometric structures at $o$ are involved. It is possible that all manifolds $M_k$, $M$ be pairwise non-homeomorphic.  In the case when the geometric structure contains a Riemannian metric (e.g. almost-hermitian structures) or univocally determines one (e.g. $G_2$-structures), the correspondence sequence $(M_k,g_k)$ has necessarily {\it bounded geometry} by homogeneity (i.e. for all $r>0$ and $j\in\ZZ_{\geq 0}$, $\sup\limits_k\;\sup\limits_{B_{g_k}(0,r)} |\nabla_{g_k}^j\Riem(g_k)|_{g_k}<\infty$).  However, the injectivity radius of the Riemannian manifolds $(M_k,g_k)$ may go to zero.

\item {\it local}: $\Omega_k$ stabilizes, i.e.\ there is a nonempty open subset $\Omega\subset\Omega_k$ for every $k$ sufficiently large.  Again, if a Riemannian metric is involved in the geometric structure, then there is a positive lower bound for the injectivity radii $\inj(M_k,g_k,o)$, which is often called the {\it non-collapsing} condition.

\item {\it pointed or Cheeger-Gromov}: $\Omega_k$ exhausts $M$, i.e. $\Omega_k$ contains any compact subset of $M$ for $k$ sufficiently large.  We note that in the homogeneous case, the location of the basepoints play no role, neither in the pointed convergence nor in the bounds considered in the items above, in the sense that we can change all of them by any other sequence of points and use homogeneity.  However, topology issues may still arise at this level of convergence.

\item {\it smooth (up to pull-back by diffeomorphisms)}: $\Omega_k=M$ and $\phi_k:M\longrightarrow M_k$ is a diffeomorphism for all $k$.  Thus $\phi_k^*\gamma_k$ converges smoothly to $\gamma$ uniformly on compact sets in $M$.  This necessarily holds for any sequence which is convergent in the pointed sense if $M=G/K$ is compact.
\end{itemize}

It follows at once from the definitions that these notions of convergence satisfy:

\begin{center}
smooth $\quad\Rightarrow\quad$ pointed $\quad\Rightarrow\quad$ local $\quad\Rightarrow\quad$ infinitesimal.
\end{center}

None of the converse assertions hold for homogeneous Riemannian manifolds (see \cite{spacehm}).  However, in the case when a Riemannian metric is involved, it is worth noticing that local convergence implies bounded geometry and non-collapsing for the sequence of metrics $g_k$ associated to the geometric structures $\gamma_k$, and thus there must exist a pointed convergent subsequence to a complete Riemannian manifold $(N,g)$ by the Compactness Theorem (see e.g. \cite{Ptr}).  In \cite[Section 7]{LtyWei}, the authors study pointed convergence of $G_2$-structures and prove a compactness result.

We may also consider convergence of the algebraic side of homogeneous spaces.  Recall from Section \ref{invgs} the space $\hca_{q,n}(\gamma)$ of Lie algebras parameterizing the set of all $n$-dimensional simply connected homogeneous spaces with $q$-dimensional isotropy endowed with an invariant geometric structure, which inherits the usual vector space topology from $\Lambda^2\ggo^*\otimes\ggo$.  We shall always denote by $\mu_k\to\lambda$ the convergence in $\hca_{q,n}(\gamma)$ relative to such topology.

The following result can be proved in much the same way as \cite[Theorem 6.12,(i)]{spacehm}.  We take this opportunity to make the following corrigenda: \cite[Theorem 6.12,(ii)]{spacehm} is false and in \cite[Corollary 6.20]{spacehm}, parts (ii)-(v) each follow from part (i), but none of the converse assertions is true.

\begin{theorem}\label{convmu}
 If $\mu_k\to\lambda$ in $\hca_{q,n}(\gamma)$, then $(G_{\mu_k}/K_{\mu_k},\gamma)$ converges to $(G_{\lambda}/K_{\lambda},\gamma)$ in the infinitesimal sense.
\end{theorem}

The converse does not hold in the Riemannian case (see \cite[Remark 6.13]{spacehm}).  As some sequences of Aloff-Walach spaces show (see \cite[Example 6.6]{spacehm}), in order to get the stronger local convergence from the usual convergence of brackets $\mu_k\to\lambda$, it is necessary (and also sufficient) to require an `algebraic' non-collapsing type condition.

\begin{definition}\label{Lie-inj}
The {\it Lie injectivity radius} of a Riemannian homogeneous space $\left(G_\mu/K_\mu,g\right)$ is the largest $r_\mu>0$ such that
$$
\pi_\mu\circ\exp_\mu: B(0,r_\mu)\longrightarrow G_\mu/K_\mu,
$$
is a diffeomorphism onto its image, where $\exp_\mu:\ggo\longrightarrow G_{\mu}$ is the Lie exponential map, $\pi_\mu:G_\mu\longrightarrow G_\mu/K_\mu$ is the usual quotient map and $B(0,r_\mu)$ denotes the euclidean ball of radius $r_\mu$ in $\pg$ relative to the fixed inner product $\ip=g(o)$.
\end{definition}

We note that $\exp_\mu$ is in general quite different from the Riemannian exponential map, unless the homogeneous space is naturally reductive.

The following is the analogue to \cite[Theorem 6.14, (ii)]{spacehm} and can be proved similarly.

\begin{theorem}\label{convmu2}
Assume that the geometric structures involved either each contains or determine a Riemannian metric.  Let $\mu_k$ be a sequence such that $\mu_k\to\lambda$ in $\hca_{q,n}(\gamma)$, as $k\to\infty$, and suppose that $\inf\limits_k r_{\mu_k}>0$, where $r_{\mu_k}$ is the Lie injectivity radius of $(G_{\mu_k}/K_{\mu_k},g)$.  Then, $(G_{\mu_k}/K_{\mu_k},\gamma)$ converges to $(G_{\lambda}/K_{\lambda},\gamma)$ in the local sense.
\end{theorem}

\begin{corollary}\label{cor-conv}
Under the hypothesis of the previous theorem, there exists a subsequence of $(G_{\mu_k}/K_{\mu_k},g)$ which converges in the pointed sense to a homogeneous Riemannian manifold locally isometric to $(G_{\lambda}/K_{\lambda},g)$.
\end{corollary}

The limit for the pointed subconvergence may depend on the subsequence, as a certain sequence of alternating left-invariant metrics on $S^3$ (Berger spheres) and $\widetilde{\Sl_2}(\RR)$ shows (see \cite[Example 6.17]{spacehm}).

\section{Geometric flows}

We consider a geometric flow of the form
$$
\dpar\gamma=q(\gamma),
$$
where $\gamma=\gamma(t)$ is a one-parameter family of geometric structures on a given differentiable manifold $M$ and $q(\gamma)$ is a tensor field on $M$ of the same type as $\gamma$ associated to each geometric structure of a given class.  Usually $q(\gamma)$ is a curvature tensor, a Laplacian or the gradient field of some natural geometric functional.  Recall that a geometric structure may be defined by a set of tensor fields $\gamma$ (e.g. an almost-hermitian structure), so in that case the geometric flow will consist of a set of differential equations, one for each tensor.

Assume first that short-time existence and uniqueness of the solutions hold, which is usually the case for $M$ compact.  Our basic assumption is that the flow is invariant by diffeomorphisms, i.e. $q(\vp^*\gamma)=\vp^*\gamma$ for any $\vp\in\Diff(M)$.  In the case of flows of hermitian (resp. compatible) metrics on a fixed complex (resp. symplectic) manifold, it is enough to assume that $q$ is invariant by the group of all bi-holomorphic maps (resp. symplectomorphisms) rather than by the whole group $\Diff(M)$.   Any solution $\gamma(t)$ starting at a $G$-invariant geometric structure on a homogeneous space $M=G/K$ will therefore remain $G$-invariant for all $t$.  Indeed, $\vp^*\gamma(t)$ is also a solution for any $\vp\in G$ starting at $\vp^*\gamma(0)=\gamma(0)$ and hence $\vp^*\gamma(t)=\gamma(t)$ for all $t$ by uniqueness.  Consequently, if we fix a reductive decomposition $\ggo=\kg\oplus\pg$ of $G/K$ as in Section \ref{invgs}, then the flow equation is equivalent to an ODE for a one-parameter family $\gamma(t)$ of $\Ad(K)$-invariant tensors on $\pg$ of the form
\begin{equation}\label{flow}
\ddt \gamma(t) = q(\gamma(t)).
\end{equation}

Alternatively, without assuming short-time existence for the original flow, one can require $G$-invariance of $\gamma(t)$ for all $t$ and obtain in this way short-time existence and uniqueness of the solutions in the class of $G$-invariant metrics.  Recall that the set of all non-degenerate structures on $\pg$ is parameterized by the homogeneous space $\Gl(\pg)/G_\gamma$ for any fixed $\gamma$ among them, and the subset of those which are $\Ad(K)$-invariant is a submanifold of it.  As the $(1,1)$-tensor field $Q$ defined by $q$ is tangent to this submanifold by \eqref{adZQ}, the solution $\gamma(t)$ to \eqref{flow} stays $\Ad(K)$-invariant for all $t$.  If the uniqueness of solutions at least holds within a class of structures containing the $G$-invariant or homogeneous ones, as for the set of complete and with bounded curvature metrics in the Ricci flow case (see \cite{ChnZhu}), then this would imply, in turn, the $G$-invariance since the solution must preserve any symmetry of the initial geometric structure by arguing as above .

In any case, short-time existence (forward and backward) and uniqueness (among $G$-invariant ones) of the solutions are guaranteed.  The need for this circular argument is due to the fact that for most of the geometric flows studied in the literature, uniqueness of a solution is still an open problem in the noncompact general case, even for the Ricci flow (see \cite{Chn}).

\subsection{Bracket flow}\label{sec-BF}
Let $G/K$ be a simply connected homogeneous space ($G$ simply connected and $K$ connected) with reductive decomposition $\ggo=\kg\oplus\pg$.  Let $\gamma(t)$ be a solution to the geometric flow \eqref{flow} starting at $\gamma:=\gamma(0)$.  Since $\gamma(t)$ is nondegenerate (see \eqref{nondeg3}), we have that
$$
\gamma(t)\in\Gl(\pg)\cdot\gamma, \qquad\forall t.
$$
If say, $\gamma(t)=h(t)^*\gamma$, for $h(t)\in\Gl(\pg)$, then for each $t$, the geometric structure $(G/K,\gamma(t))$ is equivalent to $\left(G_{\mu(t)}/K_{\mu(t)},\gamma\right)$, where
$$
\mu(t):=\left[\begin{smallmatrix} I&0\\ 0&h(t) \end{smallmatrix}\right]\cdot\lb\in\hca_{q,n}(\gamma),
$$
by arguing as in \eqref{phi-equiv}. The following natural question arises:

\begin{quote}
How does the flow look on $\hca_{q,n}(\gamma)$?
\end{quote}

More precisely,

\begin{quote}
what is the ODE a curve $\mu(t)\in\hca_{q,n}(\gamma)$ must satisfy in order to yield a solution $\left(G_{\mu(t)}/K_{\mu(t)},\gamma\right)$ to the flow \eqref{flow} up to pullback by time-dependent diffeomorphisms?
\end{quote}

\begin{remark}\label{rem6}
Recall from Example \ref{exa}, (vi) and (vii) that $\Gl(\pg)$ is replaced with $\Gl(\pg,J)\simeq\Gl_n(\CC)$ or $\Spe(\pg,\omega)\simeq\Spe(n,\RR)$ in the complex and symplectic cases, respectively.
\end{remark}

It follows from \eqref{nondeg4} that for each $t$, there exists a unique operator $Q(t):=Q(\gamma(t))\in\qg_{\gamma(t)}\subset\End(\pg)$ such that
\begin{equation}\label{Qt}
q(\gamma(t))=\theta(Q(t))\gamma(t).
\end{equation}
Consider the solution $h(t)\in\Gl(\pg)$ to the ODE system
\begin{equation}\label{hode}
\ddt h(t)=-h(t)Q(t),\qquad h(0)=I,
\end{equation}
which is defined on the same interval of time as $\gamma(t)$.  If we set $\widetilde{\gamma}(t):=h(t)^{-1}\cdot\gamma$, and $h'(t):=\ddt h(t)$, then it is easy to see that
$$
\ddt\widetilde{\gamma}(t) = -\theta(h(t)^{-1}h'(t))h(t)^{-1}\cdot\gamma = \theta(Q(t))\widetilde{\gamma}(t).
$$
Thus $\gamma(t)$ and $\widetilde{\gamma}(t)$, as curves in the differentiable manifold $\Gl(\pg)/G_\gamma=\Gl(\pg)\cdot\gamma\subset T$ of nondegenerate tensors of the same type as $\gamma$ (see Section \ref{LAgs}), satisfy the same ODE (see \eqref{flow} and \eqref{Qt}).  Since $\gamma(0)=\widetilde{\gamma}(0)=\gamma$, they must coincide by uniqueness of the solution.  Thus
\begin{equation}\label{hgamma}
\gamma(t)=h(t)^{-1}\cdot\gamma=h(t)^*\gamma,
\end{equation}
where $h(t)$ is the family of invertible maps obtained in \eqref{hode}.  In the light of the approach proposed in Section \ref{hm}, this implies that if we consider the family of Lie brackets
$$
\mu(t):=\tilde{h}(t)\cdot\lb, \qquad \tilde{h}(t):=\left[\begin{smallmatrix} I&0\\ 0&h(t) \end{smallmatrix}\right]:\ggo\longrightarrow\ggo,
$$
where $\lb$ is the Lie bracket of $\ggo$, then $\mu(t)\in\hca_{q,n}(\gamma)$ for all $t$ by Proposition \ref{const} (recall that $\gamma(t)=h(t)^*\gamma$ is $G$-invariant) and the equivariant diffeomorphism defined in \eqref{phi-equiv},
$$
\vp(t):(G/K,\gamma(t))\longrightarrow \left(G_{\mu(t)}/K_{\mu(t)},\gamma\right),
$$
is an equivalence between geometric structures (i.e.\ $\vp(t)^*\gamma=\gamma(t)$).  In particular, by \eqref{Qmuh}, $q(\gamma(t))=h(t)^*q(\gamma)$, or equivalently, $Q_{\mu(t)}=h(t)Q(t)h(t)^{-1}$, where $Q_\mu$ is defined as in \eqref{defQmu}.  Thus $h'(t)=-Q_{\mu(t)}h(t)$, and since $\ddt\mu(t) = -\delta_{\mu(t)}(\tilde{h}'(t)\tilde{h}(t)^{-1})$ (see the computation of $\ddt\lambda$ in the proof of \cite[Theorem 3.3]{homRF}), we obtain that the family $\mu(t)\in \Lambda^2\ggo^*\otimes\ggo$ of brackets satisfies the following evolution equation, called the {\it bracket flow}:
\begin{equation}\label{BF}
\ddt\mu(t)=\delta_{\mu(t)}\left(\left[\begin{smallmatrix} 0&0\\ 0&Q_{\mu(t)} \end{smallmatrix}\right]\right), \qquad\mu(0)=\lb,
\end{equation}
where $\delta_\mu:\End(\ggo)\longrightarrow\Lambda^2\ggo^*\otimes\ggo$ is minus the derivative of the $\Gl(\ggo)$-action \eqref{action} and it is given by
\begin{equation}\label{delta}
\delta_\mu(A):=\mu(A\cdot,\cdot)+\mu(\cdot,A\cdot)-A\mu(\cdot,\cdot), \qquad\forall A\in\End(\ggo).
\end{equation}

The bracket flow is therefore the answer to the questions formulated at the beginning of the section.  Equation (\ref{BF}) is well defined since $Q_\mu$ can be computed for any $\mu\in\Lambda^2\ggo^*\otimes\ggo$, not only for $\mu\in\hca_{q,n}(\gamma)$ ($Q_\mu$ is usually polynomial on $\mu$).  However, as the following lemma shows, this technicality is only needed to define the ODE.

\begin{lemma}\label{muflow}
The set $\hca_{q,n}(\gamma)$ is invariant under the bracket flow, in the sense that if $\mu_0\in\hca_{q,n}(\gamma)$, then the bracket flow solution $\mu(t)\in\hca_{q,n}(\gamma)$ for all $t$ where it is defined.  Furthermore,
$$
\mu(t)(Z,X)\equiv\mu_0(Z,X), \qquad\forall Z\in\kg,\; X\in\ggo,
$$
i.e.\ only $\mu(t)|_{\pg\times\pg}$ is actually evolving.
\end{lemma}

\begin{proof}
We must check conditions (h1)-(h3) in Definition \ref{hqn} and condition (h4) at the beginning of Section \ref{invgs} for $\mu=\mu(t)$.  We first note that for each $\mu$, the field defined as the right hand side of (\ref{BF}) is tangent to the differentiable submanifold $H\cdot\mu\subset\Lambda^2\ggo^*\otimes\ggo$, where
$$
H:=\left[\begin{smallmatrix} I&0\\ 0&\Gl(\pg)
\end{smallmatrix}\right]\subset\Gl(\ggo).
$$
Thus such a field is tangent to $H\cdot\mu_0$ at every point $\mu\in H\cdot\mu_0$, and so its integral curve $\mu(t)\in H\cdot\mu_0$ for all $t$.  This implies that condition (h1) holds for $\mu(t)$ and that $\mu(t)|_{\kg\times\kg}=\mu_0$ for all t.

Let us now prove that $\mu(t)|_{\kg\times\pg}=\mu_0$ for all t.  For each $Z\in\kg$, consider $\psi:=e^{\ad_{\mu_0}{Z}}\in\Aut(\mu_0)$.  By using that
$\psi_\pg\cdot\gamma=\gamma$ for $\psi_\pg:=\psi|_{\pg}$, one obtains from \eqref{Qmuh} that $Q_{\psi.\mu}=\psi_{\pg}Q_{\mu}(\psi_{\pg})^{-1}$ for any $\mu$.  Thus the curve
$\lambda(t):=\psi\cdot\mu(t)$ satisfies
\begin{align*}
\ddt\lambda(t)=& \psi\cdot\ddt\mu(t)=\psi\cdot\left(\delta_{\mu(t)}\left(\left[\begin{smallmatrix} 0&0\\
0&Q_{\mu(t)}\end{smallmatrix}\right]\right)\right) \\
=& \delta_{\psi\cdot\mu(t)}\left(\psi\left[\begin{smallmatrix} 0&0\\ 0&Q_{\mu(t)}
\end{smallmatrix}\right]\psi^{-1}\right) = \delta_{\lambda(t)}\left(\left[\begin{smallmatrix} 0&0\\ 0&Q_{\lambda(t)}
\end{smallmatrix}\right]\right).
\end{align*}
But $\lambda(0)=\psi\cdot\mu_0=\mu_0$, so $\lambda(t)=\mu(t)$ for all $t$ by uniqueness of the solution.  Thus $\psi\cdot\mu(t)=\mu(t)$ for all $t$, which implies that $\psi_{\pg}$ commutes with $Q_{\mu(t)}$ and hence
\begin{equation}\label{muk2}
[\ad_{\mu_0}{Z}|_{\pg},Q_{\mu(t)}]= 0,\qquad\forall Z\in\kg.
\end{equation}
It follows from (\ref{BF}) that
$$
\ddt\ad_{\mu(t)}{Z}|_{\pg} = [\ad_{\mu(t)}{Z}|_{\pg},Q_{\mu(t)}],
$$
and since the same ODE is satisfied by the constant map $\ad_{\mu_0}{Z}|_{\pg}$, it follows that
\begin{equation}\label{admuz}
\ad_{\mu(t)}{Z}|_{\pg}\equiv\ad_{\mu_0}{Z}|_{\pg}, \qquad\forall Z\in\kg,
\end{equation}
that is, $\mu(t)|_{\kg\times\pg}\equiv\mu_0$.  Conditions (h3) is therefore satisfied by $\mu(t)$ for all $t$.  Finally, since $K_{\mu_0}$ is closed in $G_{\mu_0}$, it follows that $K_{\mu(t)}$ is also closed in $G_{\mu(t)}$ as it is the image of $K_{\mu_0}$ by the isomorphism between $G_{\mu_0}$ and $G_{\mu(t)}$ with derivative at the identity given by $\left[\begin{smallmatrix} I&0\\ 0&h(t) \end{smallmatrix}\right]$ (recall that $\mu(t)\in H\cdot\mu_0$ for all $t$).  This implies that condition (h2) holds for $\mu(t)$, concluding the proof of the lemma.
\end{proof}

According to Lemma \ref{muflow}, the bracket flow equation (\ref{BF}) can be rewritten by using \eqref{delta} as the simpler system
\begin{equation}\label{BFsis}
\left\{\begin{array}{ll}
\ddt\mu_{\kg}(t)=\mu_{\kg}(Q_{\mu(t)}\cdot,\cdot)+\mu_{\kg}(t)(\cdot,Q_{\mu(t)}\cdot), & \\
& \mu_{\kg}(0)+\mu_{\pg}(0)=\mu_0|_{\pg\times\pg},\\
\ddt\mu_{\pg}(t)=\delta_{\mu_\pg(t)}(Q_{\mu(t)}), &
\end{array}\right.
\end{equation}
where $\mu_{\kg}$ and $\mu_{\pg}$ are respectively the $\kg$- and $\pg$-components of $\mu|_{\pg\times\pg}$ defined in \eqref{decmu}.

We also conclude from Lemma \ref{muflow} that a homogeneous space $\left(G_{\mu(t)}/K_{\mu(t)},\gamma\right)$ endowed with an invariant geometric structure can indeed be associated to each $\mu(t)$ in a bracket flow solution provided that $\mu_0\in\hca_{q,n}(\gamma)$ (see \eqref{hsmu-gamma}).

We are finally in a position to state and prove the main result of this section.  Let $(G/K,\gamma)$ be a simply connected homogeneous space ($G$ simply connected and $K$ connected) endowed with a $G$-invariant geometric structure $\gamma$ and a reductive decomposition $\ggo=\kg\oplus\pg$. We consider the one-parameter families
$$
(G/K,\gamma(t)), \qquad \left(G_{\mu(t)}/K_{\mu(t)},\gamma\right),
$$
where $\gamma(t)$ is the solution to the geometric flow \eqref{flow} starting at $\gamma$ and $\mu(t)$ is the solution to the bracket flow \eqref{BF} starting at the Lie bracket $\lb$ of $\ggo$, the Lie algebra of $G$.  Recall that $\ggo=\kg\oplus\pg$ is a reductive decomposition for each of the homogeneous spaces involved.

\begin{theorem}\label{BF-thm}
There exist equivariant diffeomorphisms $\vp(t):G/K\longrightarrow G_{\mu(t)}/K_{\mu(t)}$ such that
$$
\gamma(t)=\vp(t)^*\gamma, \qquad\forall t.
$$
Moreover, each $\vp(t)$ can be chosen to be the equivariant diffeomorphism determined by the Lie group isomorphism $G\longrightarrow G_{\mu(t)}$ with derivative $\tilde{h}(t):=\left[\begin{smallmatrix} I&0\\ 0&h(t) \end{smallmatrix}\right]:\ggo\longrightarrow\ggo$, where $h(t):=d\vp(t)|_o:\pg\longrightarrow\pg$ is the solution to any of the following ODE's:
\begin{itemize}
\item[(i)] $\ddt h(t)=-h(t)Q(t)$, $h(0)=I$, where $Q(t)\in\qg_{\gamma(t)}\subset\End(\pg)$ is defined by
$$
\theta(Q(t))\gamma(t)=q(G/K,\gamma(t)).
$$
%\item[ ]
\item[(ii)] $\ddt h(t)=-Q_{\mu(t)} h(t)$, $h(0)=I$, where $Q_\mu\in\qg_{\gamma}\subset\End(\pg)$ is defined by
$$
\theta(Q_\mu)\gamma=q(G_\mu/K_\mu,\gamma).
$$
\end{itemize}
The following conditions also hold:
\begin{itemize}
\item[(iii)] $\gamma(t)=h(t)^*\gamma=h(t)^{-1}\cdot\gamma$.
\item[ ]
\item[(iv)] $\mu(t)=\tilde{h}(t)\cdot\lb$.
\end{itemize}
\end{theorem}

\begin{proof}
We have already proved through \eqref{Qt}-\eqref{BF} that part (i) implies all the other statements in the theorem.  Let us now assume that part (ii) holds, and so $h(t)$ is defined on the same time interval as $\mu(t)$.  Using that
$$
\ddt\tilde{h}(t)\cdot\lb = -\delta_{\tilde{h}(t)\cdot\lb}(\tilde{h}'(t)\tilde{h}(t)^{-1}),
$$
we obtain that $\tilde{h}(t)\cdot\mu_0$ satisfies the same ODE as $\mu(t)$ and it also starts at $\lb$.  Thus $\tilde{h}(t)\cdot\lb=\mu(t)\in\hca_{q,n}(\gamma)$ for all $t$ (i.e. part (iv) holds), from which easily follows that $h(t)$ satisfies (\ref{adkh}) and therefore $\widetilde{\gamma}(t):=h(t)^*\gamma$ defines a $G$-invariant structure on $G/K$ for all $t$.  Moreover, we have that the corresponding equivariant diffeomorphism
$$
\vp(t):(G/K,\widetilde{\gamma}(t))\longrightarrow \left(G_{\mu(t)}/K_{\mu(t)},\gamma\right)
$$
is an equivalence for all $t$ (see Proposition \ref{const}), which implies that $Q_{\mu(t)}=h(t)Q(\widetilde{\gamma}(t))h(t)^{-1}$ and so $h'(t)=-h(t)Q(\widetilde{\gamma}(t))$.  Thus $\gamma(t)$ is a solution to the flow \eqref{flow} by arguing as in (\ref{hgamma}), and consequently, $\widetilde{\gamma}(t)=\gamma(t)$ for all $t$ by uniqueness.  In this way, parts (i) and (iii) follow, concluding the proof of the theorem.
\end{proof}

The following useful facts are direct consequences of the theorem:

\begin{itemize}
\item The geometric flow solution $\gamma(t)$ and the bracket flow solution $\gamma_{\mu(t)}$ differ only by pullback by time-dependent diffeomorphisms.  So the behavior of any kind of geometric quantity can be addressed on the bracket flow, which provides a useful tool to study regularity questions on the flow (see Sections \ref{evol-norm} and \ref{sec-reg} below).

\item The flows are equivalent in the following sense: each one can be obtained from the other by solving the corresponding ODE in part (i) or (ii) and applying parts (iv) or (iii), accordingly.

\item The maximal interval of time $(T_-,T_+)$ where a solution exists is therefore the same for both flows.
\end{itemize}

The above theorem has also the following application on convergence, which follows from Corollary \ref{cor-conv}.  Recall the geometric scaling $c\cdot\mu$ given in \eqref{scmu}.

\begin{corollary}\label{conv}
Let $\mu(t)$ be a bracket flow solution and assume that $c_k\cdot\mu(t_k)\to\lambda\in\hca_{q,n}(\gamma)$, for some nonzero numbers $c_k\in\RR$ and a subsequence of times $t_k\to T_\pm$.  Assume that the geometric structure involved either contains or determines a Riemannian metric and that the corresponding Lie injectivity radii satisfy $\inf\limits_k r_{c_k\cdot\mu(t_k)}>0$. Then, after possibly passing to a subsequence, the Riemannian manifolds $\left(G/K,c_k^{-2}g(t_k)\right)$ converge in the pointed (or Cheeger-Gromov) sense to a Riemannian manifold locally isometric to $(G_\lambda/K_\lambda,g_0)$, as $k\to\infty$.  Here $g(t)$ is the family of $G$-invariant metrics associated to the geometric flow solution $\gamma(t)$ on $G/K$ starting at $\gamma$.
\end{corollary}

We note that the limiting Lie group $G_\lambda$ in the above corollary might be non-isomorphic to $G$, and consequently, the limiting homogeneous space $G_\lambda/K_\lambda$ might be non-diffeomorphic, and even non-homeomorphic, to $G/K$.

At most one limit up to scaling can be obtained by considering different normalizations of the bracket flow.  More precisely, assume that $c(t)\cdot\mu(t)\to\lambda\ne 0$, as $t\to T_\pm$.  Then the limit $\tilde{\lambda}$ of any other converging normalization $a(t)\cdot\mu(t)$ necessarily satisfies $\tilde{\lambda}=c\cdot\lambda$ for some $c\in\RR$ (see \cite[Proposition 4.1,(iii)]{homRS}).  Recall that the above observation only concerns solutions which are not chaotic, in the sense that the $\omega$-limit is a single point.

\subsection{Evolution of the bracket norm}\label{evol-norm}
In addition to the direct sum $\ggo=\kg\oplus\pg$ fixed in \eqref{fixdec} and the tensor $\gamma$ fixed in \eqref{hqng}, we can also fix an inner product $\ip$ on $\ggo$ such that $\la\kg,\pg\ra=0$   The norm $|\mu|$ of a Lie bracket is therefore defined in terms of the corresponding canonical inner product on $\Lambda^2\ggo^*\otimes\ggo$ given by
\begin{equation}\label{ipg}
\la\mu,\lambda\ra:=\sum \la\mu(e_i,e_j),\lambda(e_i,e_j)\ra,
\end{equation}
where $\{ e_i\}$ is any orthonormal basis of $\ggo$.  The natural inner product on $\End(\ggo)$, $\la A,B\ra:=\tr{AB^t}$, is also determined by $\ip$.

We now compute the evolution equation for the norm $|\mu(t)|$ along a bracket flow solution $\mu(t)$.  Our motivation here is that this may be useful, for instance, to prove long-time existence.  Indeed, if $|\mu(t)|$ remained bounded then $\mu(t)$ would be defined for all $t\in[0,\infty)$ and long-time existence would follow for the original geometric flow solution $\gamma(t)$ by Theorem \ref{BF-thm}.

For each $\mu\in\Lambda^2\ggo^*\otimes\ggo$, consider the symmetric operator $\mm_{\mu_\pg}:\pg\longrightarrow\pg$ defined by
\begin{equation}\label{mm-def}
\tr{\mm_{\mu_\pg}E}=-\unc\la\delta_{\mu_\pg}(E),\mu_\pg\ra, \qquad\forall E\in\End(\pg).
\end{equation}
It follows that $M_{\mu_\pg}\in\sym(\pg)$ since for any $E\in\sog(\pg)$,
$$
\la\delta_{\mu_\pg}(E),\mu_\pg\ra = -\la\ddt\bigg|_0 e^{tE}\cdot\mu_\pg,\mu_\pg\ra = -\unm\ddt\bigg|_0|e^{tE}\cdot\mu_\pg|^2 = -\unm\ddt\bigg|_0|\mu_\pg|^2 = 0.
$$
Note that $M_{\mu_\pg}$ depends only on $\mu_\pg:\pg\times\pg\longrightarrow\pg$, the $\pg$-component of $\mu|_{\pg\times\pg}$ (see \eqref{decmu}).  It is easy to check that if $m:\Lambda^2\pg^*\otimes\pg\longrightarrow\sym(\pg)$ is the moment map for the natural action of $\Gl(\pg)$ on $\Lambda^2\pg^*\otimes\pg$ (see e.g.\ \cite{HnzSchStt} or \cite{cruzchica} and the references therein for more information on real moment maps), then
\begin{equation}\label{mmR}
m(\mu_{\pg})=\tfrac{4}{|\mu_{\pg}|^2} \mm_{\mu_{\pg}}.
\end{equation}
We collect in the following lemma three useful facts which easily follow from (\ref{mm-def}).

\begin{lemma}\label{prop}
For any $\mu\in\Lambda^2\ggo^*\otimes\ggo$, the following conditions hold:
\begin{itemize}
\item[(i)] If $\delta_{\mu_{\pg}}^t:\Lambda^2\pg^*\otimes\pg\longrightarrow\End(\pg)$ is the transpose of $\delta_{\mu_{\pg}}$, then
$$
\delta_{\mu_{\pg}}(I)=\mu_{\pg}, \qquad \delta_{\mu_{\pg}}^t(\mu_{\pg})=-4\mm_{\mu_{\pg}}.
$$
\item[(ii)] $\tr{\mm_{\mu_{\pg}}}=-\unc |\mu_{\pg}|^2$.

\item[(iii)] $\tr{\mm_{\mu_{\pg}}D}=0$ for any $D\in\Der(\mu_\pg)$.
\end{itemize}
\end{lemma}

We also need to introduce, for each $\mu\in\Lambda^2\ggo^*\otimes\ggo$, the skew-symmetric maps
$$
J_{\mu_\kg}(Z):\pg\longrightarrow\pg, \qquad \la J_{\mu_\kg}(Z)X,Y\ra=\la\mu_\kg(X,Y),Z\ra, \qquad\forall Z\in\kg, \quad X,Y\in\pg.
$$

\begin{proposition}\label{eqs}
If $\mu(t)$ is a bracket flow solution, then
$$
\ddt |\mu_\pg(t)|^2=-8\tr{Q_{\mu(t)}\mm_{\mu_\pg(t)}}, \qquad \ddt |\mu_\kg(t)|^2=-4\tr{Q_{\mu(t)}}\sum_{i=1}^q J_{\mu_\kg}(Z_i)^2,
$$
where $\{ Z_1,\dots,Z_q\}$ is an orthonormal basis of $\kg$.
\end{proposition}

\begin{proof}
It follows from \eqref{BFsis} and Lemma \ref{prop}, (i) that
$$
\ddt |\mu_\pg|^2= 2\la\ddt\mu_\pg,\mu_\pg\ra =2\la\delta_{\mu_\pg}(Q_\mu),\mu_\pg\ra = 2\la Q_\mu,\delta_{\mu_\pg}^t(\mu_\pg)\ra =-8\la Q_\mu,\mm_{\mu_\pg}\ra.
$$
To prove the evolution of $|\mu_\kg|^2$ we use orthonormal bases $\{ X_i\}$ and $\{ Z_k\}$ of $\pg$ and $\kg$, respectively, to compute using \eqref{BFsis}:
\begin{align*}
\ddt|\mu_\kg|^2 =& 2\la\ddt\mu_\kg,\mu_\kg\ra \\
=& 2\sum_{i,j}\la\mu_\kg(Q_\mu X_i,X_j),\mu_\kg(X_i,X_j)\ra + 2\sum_{i,j}\la\mu_\kg(X_i,Q_\mu X_j),\mu_\kg(X_i,X_j)\ra\\
= & 4\sum_{i,j}\la\mu_\kg(Q_\mu X_i,X_j),\mu_\kg(X_i,X_j)\ra = 4\sum_{i,j,k}\la\mu_\kg(Q_\mu X_i,X_j),Z_k\ra\la\mu_\kg(X_i,X_j),Z_k\ra \\
= & 4\sum_{i,j,k}\la J_{\mu_\kg}(Z_k)Q_\mu X_i,X_j\ra \la J_{\mu_\kg}(Z_k)X_i,X_j\ra = 4\sum_{i,k}\la J_{\mu_\kg}(Z_k)Q_\mu X_i,J_{\mu_\kg}(Z_k)X_i\ra \\
=& -4\tr{Q_\mu\sum_k J_{\mu_\kg}(Z_k)^2},
\end{align*}
concluding the proof of the proposition.
\end{proof}

\subsection{Regularity}\label{sec-reg}
In the presence of any geometric flow, a natural question is what is the simplest quantity that, as long as it remains bounded, it prevents the formation of a singularity.  In this section, as an application of the bracket flow approach developed in Section \ref{sec-BF}, we obtain a general regularity result for any invariant geometric flow solution on a homogeneous space.

Let $(T_-,T_+)$ denote the maximal interval of time existence for the bracket flow solution $\mu(t)$ (see \eqref{BF}), or equivalently, for the solution $(G/K,\gamma(t))$ to the geometric flow \eqref{flow} starting at a $G$-invariant geometric structure $(G/K,\gamma)$.  Recall that $-\infty\leq T_-<0<T_+\leq\infty$.  It follows from \eqref{BF} that
$$
\ddt|\mu|^2 \leq 2|\mu|\left|\ddt\mu\right| \leq 2C|Q_\mu||\mu|^2, \qquad \forall t,
$$
for some constant $C>0$ depending only on the norm of the representation
$$
\pi:\glg(\ggo)\longrightarrow\End(\Lambda^2\ggo^*\otimes\ggo), \qquad \pi(A)\mu:=-\delta_\mu(A).
$$
This implies that
$$
2C\int_0^s |Q_\mu| \; dt\geq \log{|\mu(s)|^2}-\log{|\mu_0|^2}, \qquad\forall s\in[0,T_+),
$$
and since $|\mu(t)|$ must blow up at a finite singularity $T_+<\infty$ (in the sense that $|\mu(t_k)|\to\infty$ for some subsequence $t_k\to T_+$), we obtain from Theorem \ref{BF-thm} that
$$
\int_0^{T_+} |Q(\gamma(t))|_t\; dt =\infty.
$$
Here the norm $|\cdot|_t$ corresponds to the inner product $\ip_t:=h(t)^*\ip$, where $h(t)\in\Gl(\pg)$ is as in Theorem \ref{BF-thm} (recall that $Q_\mu(t)=h(t)Q(\gamma(t))h(t)^{-1}$ for all $t$).  We note that if the metric $\ip$ is compatible with, or determined by the geometric structure $\gamma$ in some sense, then so is $\ip_t$ relative to $\gamma(t)$ for all $t$.

We can now use that $|j_{\gamma(t)}|_t\equiv |j_{\gamma}|$ (see \eqref{j-equiv}) to conclude that there exist a positive constant $C$ depending only on the dimension $n$ and the type $(r,s)$ of the geometric structures such that
\begin{equation}\label{qCQ}
|q(\gamma(t))|_t\geq C |Q(\gamma(t))|_t.
\end{equation}
Thus the following general regularity result in terms of the velocity of the flow follows.

\begin{proposition}\label{cor-reg}
Let $\dpar\gamma(t)=q(\gamma(t))$ be a geometric flow which is diffeomorphism invariant.  If a $G$-invariant solution $\gamma(t)$ on a homogeneous space $G/K$ has a finite singularity at $T_+$ (resp. $T_-$), then
$$
\int_0^{T_+} |q(\gamma(t))|_t\; dt =\infty \qquad \left(\mbox{resp.}\quad \int_{T_-}^0 |q(\gamma(t))|_t\; dt =\infty\right).
$$
\end{proposition}

This was proved for the pluriclosed flow on any compact manifold in \cite[Theorem 1.2]{StrTn3}, and may be considered as the analogous to N. Sesum's result on the Ricci flow (see \cite{Ssm}).  It also generalizes the result obtained in \cite[Corollary 6.2]{SCF} in the context of curvature flows for left-invariant almost-hermitian structures on Lie groups.

In addition to diffeomorphism invariance, in what follows, we shall assume that the tensor $q$ in the geometric flow equation \eqref{flow} satisfies the following scaling property for any $\gamma$:
\begin{equation}\label{scprq}
q(c\gamma) = c^\alpha q(\gamma), \qquad\forall c\in\RR^*,
\end{equation}
for some fixed $\alpha\in\RR$.  This does hold for most of the curvature tensors considered in different evolution equations in the literature.  We note that $\alpha=0$ if $q(\gamma)$ is for example the Ricci tensor of some connection associated to a metric or to an almost-hermitian structure $\gamma$, and that $\alpha=1/3$ for the Laplacian and any Dirichlet flow for $G_2$-structures (see \cite{LtyWei,WssWtt}).

Condition \eqref{scprq} is equivalent to
\begin{equation}\label{PQc}
Q(c\gamma)=c^{\alpha-1}Q(\gamma),  \qquad\forall c\in\RR^*.
\end{equation}
By using the fact observed in \eqref{scmu} that for any $c\ne 0$, the map $(I,\tfrac{1}{c}I)$ determines an equivalence of geometric structures
$$
(G_\mu/K_\mu,c^{s-r}\gamma)\longrightarrow (G_{c\cdot\mu}/K_{c\cdot\mu},\gamma),
$$
we obtain from \eqref{PQc} that the operators defined in \eqref{defQmu} satisfy
\begin{equation}\label{PQmuc}
Q_{c\cdot\mu} = c^{(r-s)(1-\alpha)}Q_\mu, \qquad\forall c\in\RR^*,
\end{equation}
for a tensor $\gamma$ of type $(r,s)$.

We now show that not only the supreme but actually $|\mu(t)|$ must converge to infinity as $t$ approaches a finite time singularity.  The proof of the following proposition is strongly based on the arguments used by R. Lafuente in \cite{Lfn} to prove that the scalar curvature controls the formation of singularities of homogeneous Ricci flows.

\begin{proposition}\label{prop-reg}
Assume that $T_+$ is finite (resp. $T_-$).  If $(r-s)(1-\alpha)>0$, then
$$
|\mu(t)|\geq\frac{C}{(T_+-t)^{\frac{1}{(r-s)(1-\alpha)}}}, \qquad \forall t\in[0,T_+)
$$
(resp. $|\mu(t)|\geq C(t-T_-)^{-1/(r-s)(1-\alpha)}$, $\forall t\in(T_-,0]$), for some positive constant $C$ depending only on $q+n$ and the geometric flow.  In particular, $|\mu(t)|\to\infty$, as $t\to T_{\pm}$.
\end{proposition}

\begin{proof}
Assume that $T_+<\infty$ (the proof for $-\infty<T_-$ is completely analogous).  It follows from \eqref{BF} that
\begin{equation}\label{reg1}
\left|\ddt\mu\right|\leq C_1|\mu|^{(r-s)(1-\alpha)+1}, \qquad\forall t,
\end{equation}
for some constants $C_1>0$ depending only on $q+n$ and the flow.  We are using here that
\begin{equation}\label{Qnmu}
Q_\mu=|\mu|^{(r-s)(1-\alpha)}Q_{\mu/|\mu|},
\end{equation}
a fact that follows by arguing as in \eqref{scmu} but for $h:=(|\mu|^{-1}I,|\mu|^{-1}I)$, and that the continuous map $\mu\mapsto Q_\mu$ attains a maximum value on the sphere of $\Lambda^2\ggo^*\otimes\ggo$ depending only on $q+n$.  This implies that
$$
\ddt|\mu|^2\leq 2C_1|\mu|^{(r-s)(1-\alpha)+2} = 2C_1(|\mu|^2)^{\frac{(r-s)(1-\alpha)}{2}+1},
$$
and so for any $t_0\in[0,T_+)$,
$$
|\mu(t)|^2\leq \left(-(r-s)(1-\alpha)C_1(t-t_0)+|\mu(t_0)|^{-(r-s)(1-\alpha)}\right)^{\frac{-2}{(r-s)(1-\alpha)}}, \qquad\forall t\in[t_0,T_+).
$$
Thus $T_+\geq t_0+\frac{|\mu(t_0)|^{-(r-s)(1-\alpha)}}{(r-s)(1-\alpha)C_1}$ since $|\mu(t)|$ must blow up at a singularity, concluding the proof of the proposition.
\end{proof}

The following corollary thus follows from \eqref{qCQ} and \eqref{Qnmu}.

\begin{corollary}\label{cor-reg}
Let $\dpar\gamma(t)=q(\gamma(t))$ be a geometric flow of type $(r,s)$ which is diffeomorphism invariant and such that the scaling property \eqref{scprq} holds for some $\alpha$, and assume that $(r-s)(1-\alpha)>0$.  If a $G$-invariant solution $\gamma(t)$ on a homogeneous space $G/K$ has a finite singularity at $T_+$ (resp. $T_-$), then
$$
|q(\gamma(t))|_t\geq\frac{C}{T_+-t}, \qquad \forall t\in[0,T_+)
$$
(resp. $|q(\gamma(t))|\geq C(t-T_-)^{-1}$, $\forall t\in(T_-,0]$), for some positive constant $C$ depending only on $q+n$ and the geometric flow.
\end{corollary}

\subsection{Self-similar solutions and soliton geometric structures}
A geometric structure $\gamma$ on a differentiable manifold $M$ will flow self-similarly along a geometric flow $\dpar\gamma=q(\gamma)$, in the sense that the solution $\gamma(t)$ starting at $\gamma$ has the form
$$
\gamma(t)=c(t)\vp(t)^*\gamma, \qquad \mbox{for some}\quad c(t)\in\RR^*, \quad \vp(t)\in\Diff(M),
$$
if and only if
$$
q(\gamma)=c\gamma+\lca_{X}\gamma, \qquad \mbox{for some}\quad c\in\RR, \quad X\in\chi(M)\; \mbox{(complete)},
$$
where $\lca_X$ denotes Lie derivative (see Remark \ref{rem3} concerning the complex and symplectic cases).  This can be proved as follows.  For $\gamma(t)=c(t)\vp(t)^*\gamma$, $c(0)=1$, $\vp(0)=id$, one has that
$$
\dpar\gamma(t) = c'(t)\vp(t)^*\gamma + c(t)\vp(t)^*\lca_{X(t)}\gamma,
$$
where $X(t)$ is the time-dependent family of vector fields generating the diffeomorphisms $\vp(t)$ (i.e. $X_{\vp(t)(p)}:=\dds|_{s=t}\vp(s)(p)$), and on the other hand, from the scaling property \eqref{scprq} and the diffeomorphism invariance of the flow, we have
$$
q(\gamma(t)) = c(t)^\alpha\vp(t)^*q(\gamma).
$$
The sufficiency therefore follows by evaluating at $t=0$ with $c=c'(0)$ and $X=X(0)=\ddt|_0\vp(t)$, and for the necessary part we can take
$$
c(t):=\left((1-\alpha)ct+1\right)^{1/(1-\alpha)},
$$
if $\alpha\ne 1$, which satisfies $c'(t)=cc(t)^\alpha$, $c(0)=1$ (and $c(t)=e^{ct}$ for $\alpha=1$) and consider the flow $\vp(t)$ generated by the time-dependent vector fields $X(t):=c(t)^{\alpha-1}X$.

In analogy to the terminology used in Ricci flow theory, in the case when $\alpha<1$, we call such $\gamma$ a {\it soliton geometric structure} and we say it is {\it expanding}, {\it steady} or {\it shrinking}, if $c>0$, $c=0$ or $c<0$, respectively.  Note that the maximal interval of existence $(T_-,T_+)$ for these self-similar solutions equals $(\tfrac{-1}{(1-\alpha)c},\infty)$, $(-\infty,\infty)$ and $(-\infty,\tfrac{-1}{(1-\alpha)c})$, respectively.

On homogeneous spaces, in view of the equivalence between any geometric flow and the corresponding bracket flow given by Theorem \ref{BF-thm}, we may also wonder about self-similarity for bracket flow solutions.  A natural way, as usual for any ODE system on a vector space, would be to consider solutions which only evolve by scaling.  From our `geometric scaling' given in \eqref{scmu}, what we obtain are bracket flow solutions of the form $\mu(t)=c(t)\cdot\mu_0$ for some $c(t)\in\RR^*$.  However, recall that $(G_{\mu(t)}/K_{\mu(t)},\gamma)$ is equivalent to $(G_{\mu_0}/K_{\mu_0},\gamma(t))$ for each $t$ (i.e. they coincide up to pull back by a diffeomorphism), and so geometrically $\mu(t)=c(t)\cdot\mu_0$ can indeed be viewed as a self-similar solution in the above sense.

\begin{theorem}\label{rsequiv}
For a simply connected homogeneous space $(G/K,\gamma)$ endowed with a $G$-invariant geometric structure $\gamma$, the following conditions are equivalent:
\begin{itemize}
\item[(i)] The bracket flow solution starting at $\lb$ is given by
$$
\mu(t)=c(t)\cdot\lb, \qquad\mbox{for some}\quad c(t)>0, \quad c(0)=1,
$$
or equivalently,
$$
\mu(t)|_{\kg\times\ggo}\equiv\lb, \qquad \mu_{\kg}(t)=c(t)^2\lb_{\kg}, \qquad \mu_{\pg}(t)=c(t)\lb_{\pg}.
$$
\item[(ii)] The operator $Q(\gamma)\in\qg_\gamma\subset\End(\pg)$ such that $\theta(Q(\gamma))=q(\gamma)$ satisfies
$$
Q(\gamma)=cI+D_\pg, \qquad \mbox{for some} \quad c\in\RR, \quad D=\left[\begin{smallmatrix} 0&0\\ 0&D_\pg \end{smallmatrix}\right] \in\Der(\ggo).
$$
\end{itemize}
In that case, if we set $a_{r,s,\alpha}:=(s-r)(1-\alpha)$, then the geometric flow solution starting at $\gamma$ is given by
$$
\gamma(t) = b(t) e^{s(t)D_\pg}\cdot\gamma,
$$
where
$$
b(t)=\left(a_{r,s,\alpha}ct+1\right)^{1/(1-\alpha)},\qquad s(t)=\tfrac{1}{a_{r,s,\alpha}c}\log(a_{r,s,\alpha}ct+1),
$$
and $(G/K,\gamma)$ is a soliton geometric structure with
$$
q(\gamma)=(s-r)c\gamma-\lca_{X_D}\gamma,
$$
where $X_D$ denotes the vector field on $G/K$ defined by the one-parameter subgroup of $\Aut(G)$ attached to the derivation $D$.
\end{theorem}

\begin{proof}
Assume first that part (i) holds.  By taking derivatives at $t=0$ we obtain that
$$
\delta_{\lb}\left(\left[\begin{smallmatrix} 0&0\\ 0& c'(0)I \end{smallmatrix}\right]\right)
= \mu'(0) = \delta_{\lb}\left(\left[\begin{smallmatrix} 0&0\\ 0&Q(\gamma)\end{smallmatrix}\right]\right),
$$
from which part (ii) follows with $D=\left[\begin{smallmatrix} 0&0\\ 0&Q(\gamma)-c'(0)I\end{smallmatrix}\right]\in\Der(\ggo)$ and $c=c'(0)$.

It is easy to see that for $\gamma(t)= b(t)e^{s(t)D_\pg}\cdot\gamma$, $b(0)=1$, $s(0)=0$, one has
$$
\ddt\gamma(t) = b'(t)e^{s(t)D_\pg}\cdot\gamma + b(t)e^{s(t)D_\pg}\cdot\theta(s'(t)D_\pg)\gamma,
$$
and on the other hand,
$$
q(\gamma(t)) = b(t)^\alpha e^{s(t)D_\pg}\cdot q(\gamma) = b(t)^\alpha e^{s(t)D_\pg}\cdot\theta(Q(\gamma))\gamma.
$$
Assume that part (ii) holds.  It follows that
$$
\theta(Q(\gamma))=c\theta(I)\gamma+\theta(D_\pg)\gamma = c(s-r)\gamma+\theta(D_\pg)\gamma,
$$
and therefore $\gamma(t)$ will be a solution as soon as $b'(t)=(s-r)cb(t)^\alpha$ and $b(t)s'(t)=b(t)^\alpha$, which hold for the functions $b(t)$ and $s(t)$ given in the proposition.  On the other hand, since $[Q(\gamma),D_\pg]=0$, we have that
$$
Q(\gamma(t))=b(t)^{\alpha-1}e^{s(t)D_\pg}Q(\gamma)e^{-s(t)D_\pg} = b(t)^{\alpha-1}Q(\gamma),
$$
and thus $h(t):=e^{-s(t)Q(\gamma)}$ satisfies $h'(t)=-h(t)Q(\gamma(t))$.  It now follows from Theorem \ref{BF-thm} that
$$
\mu(t)=\left[\begin{smallmatrix} I&0\\ 0& h(t) \end{smallmatrix}\right]\cdot\lb = \left[\begin{smallmatrix} I&0\\ 0& e^{-s(t)c}I \end{smallmatrix}\right]\cdot\left(e^{-s(t)D}\cdot\lb\right) = e^{cs(t)}\cdot\lb,
$$
which implies part (i) for $c(t)=e^{cs(t)}=\left(a_{r,s,\alpha}ct+1\right)^{a_{r,s,\alpha}}$.

It only remains to prove the last statement on solitons.  Since $D\in\Der(\ggo)$ we have that $e^{tD}\in\Aut(\ggo)$ and thus there exists $\tilde{\vp}_t\in\Aut(G)$ such that $d\tilde{\vp}_t|_e=e^{tD}$ for all $t\in\RR$.  By using that $K$ is connected and $D\kg=0$ we obtain that $\tilde{\vp}_t(K)=K$ for all $t$.  This implies that $\tilde{\vp}_t$ defines a diffeomorphism $\vp_t$ of $G/K$ by $\vp_t(uK)=\tilde{\vp}_t(u)K$ for any
$u\in G$, which therefore satisfies at the origin that $d\vp_t|_{o}=e^{tD_{\pg}}$. Let  $X_D$ denote the vector field on $G/K$ defined by the one-parameter subgroup $\{\vp_t\}\subset\Diff(G/K)$, that is, $X_D(p)=\ddt|_0\vp_t(p)$ for
any $p\in G/K$.  It follows that
\begin{equation}\label{Lder}
\lca_{X_D}\gamma(o) = \ddt\bigg|_0\vp^*_t\gamma(o) =\ddt\bigg|_0 e^{-tD_{\pg}}\cdot\gamma = -\theta(D_\pg)\gamma,
\end{equation}
but since $Q(\gamma)=cI+D_{\pg}$, we obtain that
$$
q(\gamma) = \theta(Q(\gamma))\gamma = c\theta(I)\gamma + \theta(D_\pg)\gamma = c(s-r)\gamma-\lca_{X_D}\gamma,
$$
by using that every tensor in the previous formula is $G$-invariant (recall that the flow of $X_D$ is given by automorphisms of $G$).  This concludes the proof of the theorem.
\end{proof}

The above theorem motivates the following definition.

\begin{definition}\label{as}
A homogeneous space $(G/K,\gamma)$ endowed with a $G$-invariant geometric structure $\gamma$ and a reductive decomposition $\ggo=\kg\oplus\pg$ is said to be an {\it algebraic
soliton} if there exist $c\in\RR$ and $D=\left[\begin{smallmatrix} 0&0\\ 0&D_\pg \end{smallmatrix}\right]\in\Der(\ggo)$ such that
$$
Q(\gamma)=cI+D_{\pg}.
$$
\end{definition}

\begin{remark}
We note that any simply connected algebraic soliton is a soliton geometric structure by Theorem \ref{rsequiv}.  The hypothesis of $G/K$ being simply connected is in general necessary; see \cite[Remark 4.12]{solvsolitons} for a counterexample in the Ricci flow case.
\end{remark}

\begin{remark}\label{Dkcero}
Nothing changes by allowing a derivation of the form $D=\left[\begin{smallmatrix} \ast&0\\ 0&D_\pg \end{smallmatrix}\right]\in\Der(\ggo)$ in Definition \ref{as} since $D\kg=0$ must necessarily holds.  Indeed, it follows from \eqref{adZQ} that
$$
\ad{DZ}|_\pg=[D|_\pg,\ad{Z}|_\pg]=[Q(\gamma),\ad{Z}|_\pg]=0, \qquad\forall Z\in\kg,
$$
and thus $D\kg=0$ by almost-effectiveness.
\end{remark}

\begin{remark}
There is a more general way to consider a soliton $(G/K,\gamma)$ `algebraic'; namely, when there exists a one-parameter family $\tilde{\vp}(t)\in\Aut(G)$ with $\tilde{\vp}(t)(K)=K$ such that $\gamma(t)=c(t)\vp(t)^*\gamma$ is the solution to the geometric flow $\dpar\gamma=q(\gamma)$ starting at $\gamma$ for some scaling function $c(t)>0$, where $\vp(t)\in\Diff(G/K)$ is the diffeomorphism determined by $\tilde{\vp}(t)$.  As in the Ricci flow case, such $(G/K,\gamma)$ may be called a {\it semi-algebraic soliton} (see \cite[Definition 1.4]{Jbl} and \cite[Section 3]{homRS}).  It follows from (\ref{Lder}) that if $(G/K,\gamma)$ is a semi-algebraic soliton with reductive decomposition $\ggo=\kg\oplus\pg$, then
\begin{equation}\label{alg2}
Q(\gamma)=cI+\proy_{\qg_\gamma}(D_{\pg}), \qquad\mbox{for some} \quad c\in\RR, \quad D=\left[\begin{smallmatrix} 0&\ast\\ 0&D_\pg \end{smallmatrix}\right]=\ddt|_0\tilde{\vp}(t)\in\Der(\ggo),
\end{equation}
where $\proy_{\qg_\gamma}:\glg(\pg)=\ggo_\gamma\oplus\qg_\gamma\longrightarrow\qg_\gamma$ is the usual linear projection.  Note that $\pg$ may not be $D$-invariant as for algebraic solitons.  Conversely, if condition (\ref{alg2}) holds for some reductive decomposition and $G/K$ is simply connected, then one can prove in much the same way as the last statement in Theorem \ref{rsequiv} that $\left(G/K,\gamma\right)$ is indeed a soliton with
$q(\gamma)=c(s-r)\gamma - \lca_{X_D}\gamma$.
\end{remark}

\subsection{Lie group case}\label{lgcase}
Our aim in this section is to go over again the case of left-invariant geometric structures on Lie groups, i.e.\ $\hca_{0,n}(\gamma)$, the one which has been mostly applied in the literature (cf.\ for example \cite{nilricciflow, homRF, SCF} and the references therein).  Recall from Example \ref{ex0-n} that $\hca_{0,n}$ is simply the variety $\lca_n$ of $n$-dimensional Lie algebras, and since (h4) does not either give any restriction here, we obtain that $\hca_{0,n}(\gamma)=\lca_n$ and we identify
$$
\mu\in\lca_n\longleftrightarrow (G_\mu,\gamma_\mu)=(G_\mu,\gamma),
$$
where $\gamma_\mu$ denotes the left-invariant metric on the simply connected Lie group $G_\mu$ determined by the fixed tensor $\gamma$ we have on the Lie algebra $(\ggo,\mu)$ of $G_\mu$.  Condition \eqref{adkh} also holds trivially here, so every $h\in\Gl_n(\ggo)$ defines a Lie group isomorphism which is a geometric  equivalence
$$
(G_{h.\mu},\gamma)\longrightarrow(G_\mu,h^*\gamma).
$$
The orbit $\Gl(\ggo)\cdot\mu\subset\lca_n$ therefore parameterizes the set of all left-invariant structures on $G_\mu$ and the orbit $G_\gamma\cdot\mu$ parameterizes the subset of those which are equivalent to $(G_\mu,\gamma)$ via an automorphism.

We note that $\ggo=\pg$ and $\mu=\mu_\pg$ in this case, thus $\widetilde{h}(t)=h(t)$ in Theorem \ref{BF-thm} and the formulas and notation in Proposition \ref{eqs} and Theorem \ref{rsequiv} simplify considerably.

The following lower bound for the Lie injectivity radius gives rise to special convergence features for Lie groups which are not valid for homogeneous spaces in general.  Recall that $\mu\in\lca_n$ is said to be {\it completely solvable} if all the eigenvalues of $\ad_\mu{X}$ are real for any $X\in\ggo$.  It is well known that the exponential map of any simply connected completely solvable Lie group is a diffeomorphism.

\begin{lemma}\label{lieinjLG}\cite[Lemma 6.19]{spacehm}
Let $r_\mu$ be the Lie injectivity radius of $\mu\in\lca_n=\hca_{0,n}$. Then,
\begin{itemize}
\item[(i)] $r_\mu\geq\tfrac{\pi}{|\mu|}$.

\item[(ii)] $r_\mu=\infty$ for any completely solvable $\mu$ (in particular, $G_\mu$ is diffeomorphic to $\RR^n$).
\end{itemize}
\end{lemma}

We can therefore rephrase Corollary \ref{cor-conv} in the case of Lie groups in a stronger way as follows.

\begin{proposition}\label{convmu4}
Let $\mu_k$ be a sequence in $\lca_n=\hca_{0,n}$ such that $\mu_k\to\lambda$.
\begin{itemize}
\item[(i)] $\lambda\in\lca_n$.

\item[(ii)] $(G_{\mu_k},\gamma)$ converges in the local sense to $(G_{\lambda},\gamma)$.

\item[(iii)] If either $G_\lambda$ is compact or all $\mu_k$ are completely solvable, then $(G_{\mu_k},\gamma)$ smoothly converges to $(G_{\lambda},\gamma)$ up to pull-back by diffeomorphisms.

\item[(iii)] $\gamma_{\mu_k}\to \gamma_\lambda$ smoothly on $\RR^n\equiv G_{\mu_k}$, provided all $\mu_k$ are completely solvable.
\end{itemize}
\end{proposition}

\section{Overview of applications in the literature}\label{exa-sec}

The approach that proposes to vary Lie brackets rather than metrics or geometric structures has been used for decades in homogeneous geometry.  In what follows, we review some selected examples and applications in the literature, in a chronological way on each topic.  For a more complete study, we refer the reader to the references in the cited articles.  In most of these applications, geometric invariant theory of the variety of Lie algebras, including moment maps, closed orbits, stability, categorical quotients, Kirwan stratification, etc., has been exploited in one way or another.

\subsection{Pinching curvature conditions}

\begin{itemize}
\item \cite{Hnt} Classification of Lie groups admitting a metric of negative sectional curvature.

\item \cite{Mln} Scalar, Ricci and sectional curvature properties of Lie groups.

\item \cite{inter} Degenerations of $3$-dimensional real Lie algebras.

\item \cite{Ebr} Ricci curvature of $2$-step nilpotent Lie groups.

\item \cite{NklNkn, Nkl2} Existence of Ricci negative metrics on solvable Lie groups.

\item \cite{Wll4} Existence of Ricci negative metrics on some Lie groups with a compact Levi factor.  
\end{itemize}

\subsection{Einstein solvmanifolds and nilsolitons}

\begin{itemize}
\item \cite{Hbr} Foundational structure and uniqueness results on {\it Einstein solvmanifolds} (i.e.\ Einstein left-invariant metrics on solvable Lie groups).
\item \cite{soliton} Introduction of {\it nilsolitons} (i.e.\ algebraic Ricci solitons on nilpotent Lie groups), uniqueness, variational characterization and relationship with Einstein solvmanifolds.

\item \cite{Wll1,Frn1,Frn3} Classification of Einstein solvmanifolds and nilsolitons in low dimensions.

\item \cite{GrdKrr,einsteinsolv,Nkl,Wll2,Wll3,Nkl3,Jbl0,Pyn2} Structure and classification of Einstein solvmanifolds and nilsolitons.

\item \cite{standard} Proof of the standard property for Einstein solvmanifolds.

\item \cite{cruzchica} Survey on Einstein solvmanifolds and nilsolitons up to April 2008.

\item \cite{nonsingular} Nonsingular $2$-step nilpotent Lie algebras: Pfaffian forms, classification and nilsolitons.
\end{itemize}

\subsection{Ricci flow}

\begin{itemize}
\item \cite{Gzh,Pyn,nilricciflow} Ricci flow for {\it nilmanifolds} (i.e.\ left-invariant metrics on nilpotent Lie groups).

\item \cite{GlcPyn} Ricci flow evolution of $3$-dimensional homogeneous geometries.

\item \cite{homRF} Ricci flow on homogeneous spaces, after a study of different kinds of convergence of homogeneous Riemannian manifolds in \cite{spacehm} introducing the space $\hca_{q,n}$ of homogeneous spaces.

\item \cite{Arr} Ricci flow of {\it almost-abelian} solvmanifolds (i.e.\ solvable Lie groups with a codimension one abelian normal subgroup).

\item \cite{Lfn} Scalar curvature controls the formation of singularities of homogeneous Ricci flows.
\end{itemize}

\subsection{Homogeneous Ricci solitons}

\begin{itemize}
\item \cite{solvsolitons} Structure, uniqueness and classification of {\it solvsolitons} (i.e.\ algebraic Ricci solitons on solvable Lie groups).

\item \cite{Jbl} Any homogeneous Ricci soliton is isometric to a semi-algebraic soliton.  Ricci solitons under transitive semisimple and solvable Lie groups.

\item \cite{homRS} Bracket flow evolution of invariant Ricci solitons on homogeneous spaces, including a geometrical characterization of algebraic solitons as those for which the Ricci flow solution is simultaneously diagonalizable.

\item \cite{alek} Structural results on homogeneous Ricci solitons, providing new insights into the longstanding {\it Alekseevskii conjecture}: any connected homogeneous Einstein manifold of negative scalar curvature is diffeomorphic to a Euclidean space.

\item \cite{Jbl2} Any homogeneous Ricci soliton is isometric to an algebraic soliton.

\item \cite{ArrLfn1,ArrLfn2} Classification of homogeneous Ricci solitons and the Alekseevskii conjecture in low dimensions.

\item \cite{JblPtrWll} Linear and dynamically stability of algebraic Ricci solitons on many classes of Lie groups, including an open set of two-step solvsolitons, all two-step nilsolitons and all nilsolitons of dimensions six or less.

\item \cite{JblPtr} A refinement of the structure result in \cite{alek}.
\end{itemize}

\subsection{Curvature flows on almost-hermitian Lie groups}

\begin{itemize}
\item \cite{EnrFnVzz} Long-time existence for any pluriclosed flow solution on a nilpotent Lie group, as an application of the bracket flow approach.

\item \cite{SCF} General curvature flows on almost-hermitian Lie groups and their algebraic solitons.

\item \cite{Frn2} Existence of solitons for the symplectic curvature flow on any $2$- and $3$-step nilpotent Lie group of dimension $6$.

\item \cite{SCFmuA} Symplectic curvature flow and algebraic solitons on three large classes of almost-K\"ahler Lie groups: almost-abelian solvmanifolds, a construction attached to each left-symmetric algebra (which provides intriguing examples of shrinking solitons) and $4$-dimensional solvable Lie groups.

\item \cite{CRF} Chern-Ricci flow and algebraic solitons on hermitian Lie groups, including a complete picture in the almost-abelian case and a Chern-Ricci soliton hermitian metric on most of the complex surfaces which are solvmanifolds, where an unexpected shrinking soliton appeared. 
    
\item \cite{LF} Laplacian flow of closed $G_2$-structures and its algebraic solitons on homogeneous spaces.  
\end{itemize}

\section{Chern-Ricci flow}\label{CRF-sec}

The {\it Chern-Ricci flow} (CRF for short) is the evolution equation for a one-parameter family $g(t)$ of hermitian metrics on a fixed complex manifold $(M,J)$ defined by
\begin{equation}\label{CRF-intro}
\dpar g(t)=-2p(t)(\cdot,J\cdot), \qquad\mbox{or equivalently}, \qquad \dpar \omega(t)=-2p(t),
\end{equation}
where $\omega=g(J\cdot,\cdot)$ and $p=p(J,g)$ is the Chern-Ricci form (see \cite{TstWnk} for further information on this flow).  The $2$-form $p$ is closed, of type $(1,1)$, locally exact and in the K\"ahler case coincides with the Ricci form $\ricci(J\cdot,\cdot)$, so CRF becomes K\"ahler-Ricci flow as soon the starting metric is K\"ahler.  The CRF on Lie groups has already been studied in \cite{CRF} using the bracket flow approach.  Our aim in this section is to show that most of the results obtained in \cite{CRF} for Lie groups are still valid on homogeneous spaces.

Let $(G/K,J)$ be a homogeneous space endowed with a $G$-invariant complex structure $J$, which will be fixed from now on.  Given a reductive decomposition $\ggo=\kg\oplus\pg$ for $G/K$, it can be proved in much the same way as in \cite[Proposition 4.1]{Vzz2} (see also \cite{Pk}) that the Chern-Ricci form of any $G$-invariant hermitian metric $g$ on the homogeneous complex manifold $(G/K,J)$ is given by
\begin{equation}\label{CRform}
p(X,Y)=-\unm\tr{J\ad_\pg{[X,Y]_\pg}} + \unm\tr{\ad_\pg{J[X,Y]_\pg}}, \qquad\forall X,Y\in\pg,
\end{equation}
where $\ad_\pg{X}(Y):=[X,Y]_\pg$ for all $X,Y\in\pg$.  Remarkably, $p$ only depends on $J$, it is independent from the metric $g$.  This implies that along the CRF-solution $g(t)$ starting at a $G$-invariant hermitian metric $g_0$ on $(G/K,J)$, the Chern-Ricci form $p(t)\equiv p_0:=p(J,g_0)$, and so $g(t)$ is simply given by
\begin{equation}\label{CRF-sol}
g(t)=g_0-2tp_0(\cdot,J\cdot), \qquad\mbox{or equivalently}, \qquad \omega(t)=\omega_0-2tp_0.
\end{equation}
The {\it Chern-Ricci operator} $P\in\End(\pg)$ of the hermitian manifold $(G/K,J,g)$ is the hermitian map defined by
$p=\omega(P\cdot,\cdot)=g(JP\cdot,\cdot)$.  We note that by \eqref{CRF-sol}, the solution exists as long as the hermitian map $I-2tP_0$ is positive, where $P_0$ denotes the Chern-Ricci operator of $g_0$, so the maximal interval of time existence $(T_-,T_+)$ of $g(t)$ can be easily computed in terms of the extremal eigenvalues of the symmetric operator $P_0$ as follows:
\begin{equation}\label{CRFint}
T_+=\left\{\begin{array}{lcl} \infty, & \quad\mbox{if}\;P_0\leq 0, \\ \\ 1/(2p_+), & \quad\mbox{otherwise,}\end{array}\right. \qquad
T_-=\left\{\begin{array}{lcl} -\infty, & \quad\mbox{if}\;P_0\geq 0, \\ \\ 1/(2p_-), & \quad\mbox{otherwise,}\end{array}\right.
\end{equation}
where $p_+$ is the maximum positive eigenvalue of the Chern-Ricci operator $P_0$ of $g_0$ and $p_-$ is the minimum negative eigenvalue.

Since the velocity $q(g)$ of the CRF equals $-2p(\cdot,J\cdot)$  (see \eqref{CRF-intro}), we obtain from Example \ref{exa}, (vi), that
$$
Q(g)=P\in\qg_g=\herm(\pg,J,g),
$$
and thus the bracket flow is given by
\begin{equation}\label{BF-CRF}
\ddt\mu(t)=\delta_{\mu(t)}\left(\left[\begin{smallmatrix} 0&0\\ 0&P_{\mu(t)} \end{smallmatrix}\right]\right), \qquad\mu(0)=\lb,
\end{equation}
where $P_\mu\in\herm(\pg,J,g_0)$ is the Chern-Ricci operator of the hermitian homogeneous space $(G_\mu/K_\mu,J,g_0)$.  It follows from \eqref{CRF-sol} that the Chern-Ricci operator of a CRF-solution $g(t)$ equals $P(t)=(I-2tP_0)^{-1}P_0$.

The solution $h(t)\in\Gl(\pg,J)$ to the ODE in Theorem \ref{BF-thm}, (i) is therefore given by $h(t)=(I-2tP_0)^{1/2}$, from which follows that the bracket flow solution is
$$
\mu_\pg(t)=(I-2tP_0)^{1/2}\cdot\lb_\pg, \qquad \mu_\kg(t)=\left[(I-2tP_0)^{-1/2}\cdot,(I-2tP_0)^{-1/2}\cdot\right]_\kg.
$$
Thus with respect to any orthonormal basis $\{ X_1,\dots,X_{2n}\}$ of $(\pg,g_0)$ of eigenvectors of $P_0$, say with eigenvalues $\{ p_1,\dots,p_{2n}\}$, the structure coefficients of $\mu(t)|_{\pg\times\pg}$ are
\begin{equation}\label{muijk}
\begin{array}{c}
(\mu_\pg)_{ij}^k(t)=\left(\frac{1-2tp_k}{(1-2tp_i)(1-2tp_j)}\right)^{1/2} c_{ij}^k, \qquad
(\mu_\kg)_{ij}^l(t)=\left(\frac{1}{(1-2tp_i)(1-2tp_j)}\right)^{1/2} c_{ij}^l,
\end{array}
\end{equation}
where $c_{ij}^k$ are the structure coefficients of the Lie bracket $\lb$ of $\ggo$: $[X_i,X_j]_\pg=\sum c_{ij}^kX_k$; $[X_i,X_j]_\kg=\sum c_{ij}^lZ_l$, $\{ Z_l\}$ any basis of $\kg$.

A straightforward analysis using \eqref{muijk} gives that $\mu(t)$ converges as $t\to T_+$ if and only if $T_+=\infty$ (i.e. $P_0\leq 0$) and $\pg_0:=\Ker{P_0}$ satisfies $[\pg_0,\pg_0]_\pg\subset\pg_0$.  Moreover, the following conditions are equivalent in the case when $T_+=\infty$:

\begin{itemize}
\item[(i)] $\mu(t)\to 0$, as $t\to\infty$.
\item[(ii)] $[\pg,\pg_0]_\pg=0$.
\item[(iii)] $(2t+1)^{1/2}\mu(t)$ converges as $t\to\infty$.
\end{itemize}

Following the lines of \cite[Section 5]{CRF}, one can obtain many results on convergence from Corollary \ref{conv} and Section \ref{converg}, including some information on to what extent the Chern-Ricci form and the homogeneous space structure of the pointed limit are determined by the starting hermitian manifold $(G/K,J,g_0)$.  For instance, the following can be proved:

\begin{itemize}
\item If $P_0\leq 0$ (i.e. $T_\infty=\infty$) and $[\pg,\pg_0]_\pg=0$, then $g(t)/t$ converges in the pointed sense, as $t\to\infty$, to a Chern-Ricci soliton
    $(G_\infty/K_\infty,J,g_0)$ with reductive decomposition $\ggo=\kg\oplus(\pg_0^\perp\oplus\pg_0)$ and Lie bracket $\lb_\infty$ such that $\lb_\infty|_{\kg\times\pg}=\lb$ and
    $$
    [\pg_0^\perp,\pg_0^\perp]_\infty\subset\pg_0^\perp, \quad [\pg_0^\perp,\pg_0]_\infty\subset\pg_0, \quad [\pg_0,\pg_0]_\infty=0.
    $$
    The Chern-Ricci operator $P_\infty$ of the soliton satisfies $P_\infty|_{\pg_0^\perp}=-I$, $P_\infty|_{\pg_0}=0$.  In particular, if the starting hermitian metric $g_0$ has negative Chern-Ricci tensor $p_0(\cdot,J\cdot)$, then $g(t)/t$ flows to a homogeneous hermitian manifold with $p=-\omega$.

\item If the eigenspace $\pg_m$ of the maximum positive eigenvalue of $P_0$ satisfies $0\ne [\pg_m,\pg_m]_\pg\subset\pg_m$, then $T_+<\infty$ and $g(t)/(T_+-t)$ converges in the pointed sense, as $t\to T_+$, to a Chern-Ricci soliton $(G_+/K_+,J,g_0)$ with reductive decomposition $\ggo=\kg\oplus(\pg_m\oplus\pg_m^\perp)$ and Lie bracket $\lb_+$ such that $\lb_+|_{\kg\times\pg}=\lb$ and
    $$
    [\pg_m,\pg_m]_+\subset\pg_m, \quad [\pg_m,\pg_m^\perp]_+\subset\pg_m^\perp, \quad [\pg_m^\perp,\pg_m^\perp]_+=0.
    $$
    The Chern-Ricci operator of $(G_+/K_+,J,g_0)$ is given by $P_+|_{\pg_m}=\unm I$, $P_+|_{\pg_m^\perp}=0$.
\end{itemize}

\begin{remark}
The Chern scalar curvature
$$
\tr{P(t)}=\sum_{i=1}^{2n} \frac{p_i}{1-2tp_i},
$$
is strictly increasing unless $P(t)\equiv 0$ (i.e. $g(t)\equiv g_0$) and the integral of $\tr{P(t)}$ must blow up at a finite-time singularity $T_+<\infty$.  However, $\tr{P(t)}\leq \frac{C}{T_+-t}$ for some constant $C>0$, which is the claim of a well-known general conjecture for the K\"ahler-Ricci flow (see e.g. \cite[Conjecture 7.7]{SngWnk}).
\end{remark}

\section{Laplacian flow for $G_2$-structures}\label{G2sol}

The following natural geometric flow for $G_2$-structures on a $7$-dimensional manifold $M$ (see Example \ref{exa}, (v)), called the {\it Laplacian flow}, was introduced by R. Bryant in \cite{Bry}:
$$
\dpar\vp(t) = \Delta_t\vp(t),
$$
where $\Delta_t:=\Delta_{g_t}$ is the Hodge Laplacian operator of the Riemannian metric $g_t:=g_{\vp(t)}$ determined by $\vp(t)$ (i.e.\ $\Delta_t=-d\ast_t d\ast_t +\ast_td\ast_td$, where $\ast_t$ is the Hodge star operator defined by the metric $g_t$ and orientation).  We refer the reader to \cite{LtyWei} and the references therein for further information on this flow.

For each $x,y\in\RR$, consider the $7$-dimensional nilpotent Lie algebra  $\ngo=\ngo(x,y)$ with basis $\{ e_1,\dots,e_7\}$ and Lie bracket $\mu=\mu_{x,y}$ defined by
$$
\mu(e_1,e_2)=-xe_5, \quad \mu(e_1,e_3)=-ye_6; \quad\mbox{or equivalently}, \quad d_\mu e^5=xe^{12}, \quad d_\mu e^6=ye^{13}.
$$
The $3$-form
$$
\vp=e^{147}+e^{267}+e^{357}+e^{123}+e^{156}+e^{245}-e^{346},
$$
is positive and so it determines a left-invariant $G_2$-structure $\vp$ on the simply connected Lie group $N$ with Lie algebra $\ngo$.  It is easy to check that $d_\mu\vp=(y-x)e^{1237}$, which implies that $\vp$ is closed (or calibrated) if and only if $x=y$.

We ask ourselves whether $(N,\vp)$ is a {\it Laplacian soliton}, i.e.\ a soliton $G_2$-structure for the Laplacian flow. In the light of Theorem \ref{rsequiv}, it would be enough to find a derivation $D\in\Der(\ngo)$ such that
\begin{equation}\label{solg2}
\Delta_\vp\vp=k\vp+\lca_{X_D}\vp, \qquad\mbox{for some}\quad k\in\RR,
\end{equation}
where $X_D$ is the vector field on $N$ defined by the one-parameter subgroup of automorphisms $F_t$ with derivative $e^{tD}\in\Aut(\ngo)$ for all $t$.  Note that our fixed basis $\{ e_i\}$ is conveniently orthonormal (and oriented) with respect to the metric $g_\vp$, so $\Delta_\vp=\Delta_\mu=-d_\mu\ast d_\mu\ast+\ast d_\mu\ast d_\mu$.

We propose, with a certain amount of optimism, a diagonal
$$
D:=\Diag(a,b,c,d,a+b,a+c,e)\in\Der(\ngo),
$$
in terms of the basis $\{ e_i\}$.  By a straightforward computation we obtain that
$$
\Delta_\vp\vp=(x+y)e^{123} + y(x-y)e^{267} + x(x-y)e^{357},
$$
and, on the other hand by \eqref{Lder} that
\begin{align*}
\lca_{X_D}\vp =& \vp(D\cdot,\cdot,\cdot) +\vp(\cdot,D\cdot,\cdot) +\vp(\cdot,\cdot,D\cdot) \\
=& (a+d+e)e^{147} +(a+b+c+e)e^{267} +(a+b+c+e)e^{357} +(a+b+c)e^{123} \\
&+(3a+b+c)e^{156} +(a+2b+d)e^{245} +(a+2c+d)e^{346}.
\end{align*}

It follows that \eqref{solg2} can hold only if $y(x-y)=x(x-y)$, that is, $x=y$.  In fact, if we set $x=y=1$, it can be easily checked that the derivation $D=-\Diag(1,1,1,2,2,2,2)$ solves the soliton equation \eqref{solg2} with $k=5$ and provides us with an expanding Laplacian soliton $(N,\vp)$ which is closed.  Note that $\vp$ is far from being an eigenvector of $\Delta_\vp$.

The following remarks are in order.

\begin{itemize}
\item The Lie group $N$ is diffeomorphic to $\RR^7$ and it admits a cocompact discrete subgroup $\Gamma$.  However, the corresponding closed $G_2$-structure on the compact manifold $M=N/\Gamma$ is not necessarily a Laplacian soliton since the vector field $X_D$ does not descend to $M$.

\item The Laplacian flow solution $\vp(t)$ on $M=N/\Gamma$ starting at $\vp$ remains locally equivalent to $\vp$, is immortal and has apparently no chances to converge in any reasonable sense.  However, the norm of the intrinsic torsion $T$ of $\vp(t)$ converges to zero, as $t\to\infty$.

\item On the other hand, it follows from Theorem \ref{rsequiv} and Proposition \ref{convmu4} that the solution on the cover $N$ smoothly converges to the flat $(\RR^7,\vp)$ up to pull-back by time-dependent diffeomorphisms, as $t\to\infty$, uniformly on compact sets of $N=\RR^7$.  Laplacian flow evolution of $G_2$-structures on homogeneous spaces and their solitons is the subject of the forthcoming paper \cite{LF}.

\item $X_D$ is not the gradient field of any real smooth function on $N$, so $\vp$ is not a gradient soliton.

\item The metric $g_\vp$ is a Ricci soliton.  This was proved in \cite{FrnFinMnr}, where existence of closed $G_2$-structures on nilpotent Lie groups inducing Ricci solitons is studied, as well as the Laplacian flow evolution of such structures.  The solution $\vp(t)$ was explicitly given in
    \cite[Theorem 4.2]{FrnFinMnr}, and previously in \cite[Section 6.2.1, Example 2]{Bry}, though the fact that $\vp(t)$ is a self-similar solution was not mentioned in these papers.

\item Existence and uniqueness of closed Laplacian solitons on some nilpotent Lie groups admitting a closed $G_2$-structure is studied in \cite{Ncl}.
\end{itemize}


\begin{thebibliography}{MMM}
\bibitem[A]{Arr} {\sc R. Arroyo}, The Ricci flow in a class of solvmanifolds, {\it Diff. Geom. Appl.} {\bf 31} (2013), 472-485.

\bibitem[AL1]{ArrLfn1}{\sc R. Arroyo, R. Lafuente}, Homogeneous Ricci solitons in low dimensions, {\it Int. Math. Res. Notices} {\bf 2015} (2015), 4901-4932.  

\bibitem[AL2]{ArrLfn2} {\sc R. Arroyo, R. Lafuente}, The Alekseevskii conjecture in low dimensions, preprint 2015 (arXiv).

\bibitem[B]{Bry} {\sc R. Bryant}, Some remarks on $G_2$-structures, Proc. Gökova Geometry-Topology Conference (2005), 75-109.

\bibitem[C]{Chn} {\sc B-L Chen}, Strong uniqueness of the Ricci flow, {\it J. Diff. Geom.} {\bf 82} (2009), 336-382.

\bibitem[CZ]{ChnZhu} {\sc B-L Chen, X-P Zhu}, Uniqueness of the Ricci flow on complete noncompact Riemannian manifolds, {\it J. Diff. Geom.} {\bf 74} (2006), 119-154.

\bibitem[E]{Ebr} {\sc P. Eberlein}, Riemannian $2$-step nilmanifolds with prescribed Ricci tensor, {\it Contemp. Math.} {\bf 469} (2008), 167-195.

\bibitem[EFV]{EnrFnVzz} {\sc N. Enrietti, A. Fino, L. Vezzoni}, The pluriclosed flow on nilmanifolds and tamed symplectic flow, {\it J. Geom. Anal.}, in press (arXiv).

\bibitem[FFM]{FrnFinMnr} {\sc M. Fern\'andez, A. Fino, V. Manero}, Laplacian flow of closed $G_2$-structures inducing nilsolitons,  {\it J. Geom. Anal.}, in press (arXiv).

\bibitem[FC1]{Frn1} {\sc E. Fern\'andez-Culma}, Classification of $7$-dimensional Einstein Nilradicals, {\it Transf. Groups} {\bf 17} (2012), 639-656.

\bibitem[FC2]{Frn2} {\sc E. Fern\'andez-Culma}, Soliton almost K\"ahler structures on $6$-dimensional nilmanifolds for the symplectic curvature flow,
{\it J. Geom. Anal.}, in press (arXiv).

\bibitem[FC3]{Frn3} {\sc E. Fern\'andez-Culma}, Classification of Nilsoliton metrics in dimension seven, {\it Journal of Geometry and Physics} {\bf 86} (2014) 164-179.

\bibitem[GP]{GlcPyn} {\sc D. Glickenstein, T. Payne}, Ricci flow on three-dimensional, unimodular metric Lie algebras, {\it Comm. Anal. Geom.} {\bf 18} (2010), 927-962.

\bibitem[GK]{GrdKrr} {\sc C. Gordon, M. Kerr}, New homogeneous Einstein metrics of negative Ricci curvature, {\it Ann. Global Anal. and Geom.}, {\bf 19} (2001), 1-27.

\bibitem[G]{Gzh} {\sc G. Guzhvina}, The action of the Ricci flow on almost flat manifolds, Ph.D. thesis, Universit\"at M\"unster (2007).

\bibitem[H]{Hbr} {\sc J. Heber}, Noncompact homogeneous Einstein spaces, {\it Invent. math}. {\bf 133} (1998), 279-352.

\bibitem[Hn]{Hnt} {\sc H. Heintze}, On homogeneous manifolds of negative curvature, {\it Math. Ann.} {\bf 211} (1974), 23–34.

\bibitem[HSS]{HnzSchStt} {\sc P. Heinzner, G. W. Schwarz, H. St\"otzel}, Stratifications with respect to actions of real reductive groups, {\it Compositio Math.} \textbf{144} (2008), 163-185.

\bibitem[J1]{Jbl0}  {\sc M. Jablonski}, Concerning the existence of Einstein and Ricci soliton metrics on solvable Lie groups,
{\it Geom. Topol.} {\bf 15} (2011), 735-764.

\bibitem[J2]{Jbl}  {\sc M. Jablonski}, Homogeneous Ricci solitons, {\it J. reine angew. Math.} {\bf 699} (2015) 159-182.

\bibitem[J3]{Jbl2} {\sc M. Jablonski}, Homogeneous Ricci solitons are algebraic, {\it Geom. Topol.} {\bf 18} (2014), 2477-2486.

\bibitem[JP]{JblPtr} {\sc M. Jablonski, P. Petersen}, A step towards the Alekseevskii conjecture, preprint 2014 (arXiv).

\bibitem[JPW]{JblPtrWll} {\sc M. Jablonski, P. Petersen, M.B. Williams}, Linear stability of algebraic Ricci solitons, {\it J. reine angew. Math.}, in press (arXiv).

\bibitem[K]{Krg} {\sc S. Karigiannis}, Flows of $G_2$-Structures I, {\it Quart. J. Math.} {\bf 60} (2009), 487-522.

\bibitem[L]{Lfn} {\sc R. Lafuente}, Scalar curvature behavior of homogeneous Ricci flows, {\it J. Geom. Anal.}, in press (arXiv).

\bibitem[LL1]{homRS} {\sc R. Lafuente, J. Lauret}, On homogeneous Ricci solitons, {\it Quart. J. Math.} {\bf 65} (2014), 399-419.

\bibitem[LL2]{alek} {\sc R. Lafuente, J. Lauret}, Structure of homogeneous Ricci solitons and the Alekseevskii conjecture, {\it J. Diff. Geom.} {\bf 98} (2014) 315-347.

\bibitem[L1]{soliton} {\sc J. Lauret}, Ricci soliton homogeneous nilmanifolds, {\it Math. Ann.} \textbf{319} (2001), 715-733.

\bibitem[L2]{inter}  {\sc J. Lauret}, Degenerations of Lie algebras and geometry of Lie groups,
{\it Diff. Geom. Appl}. {\bf 18} (2003), 177-194.

\bibitem[L3]{cruzchica} {\sc J. Lauret},  Einstein solvmanifolds and nilsolitons, {\it Contemp. Math.} {\bf 491} (2009), 1-35.

\bibitem[L4]{standard}  {\sc J. Lauret}, Einstein solvmanifolds are standard, {\it Ann. of Math.} {\bf 172} (2010), 1859-1877.

\bibitem[L5]{solvsolitons} {\sc J. Lauret}, Ricci soliton solvmanifolds, {\it J. reine angew. Math.} {\bf 650} (2011), 1-21.

\bibitem[L6]{nilricciflow}  {\sc J. Lauret}, The Ricci flow for simply connected nilmanifolds, {\it Comm. Anal. Geom.} {\bf 19} (2011), 831-854.

\bibitem[L7]{spacehm}  {\sc J. Lauret}, Convergence of homogeneous manifolds, {\it J. London Math. Soc.} {\bf 86} (2012), 701-727.

\bibitem[L8]{homRF}  {\sc J. Lauret}, Ricci flow of homogeneous manifolds, {\it Math. Z.} {\bf 274} (2013), 373-403.

\bibitem[L9]{SCF}  {\sc J. Lauret}, Curvature flows for almost-hermitian Lie groups, {\it Transactions Amer. Math. Soc.} {\bf 367} (2015), 7453–7480.

\bibitem[L10]{LF}  {\sc J. Lauret}, Laplacian flow of $G_2$-structures on homogeneous spaces, in preparation.

\bibitem[LO]{nonsingular}  {\sc J. Lauret, D. Oscari}, On nonsingular $2$-step nilpotent Lie algebras, {\it Math. Res. Lett.} {\bf 21} (2014), 553–583.

\bibitem[LR]{CRF}  {\sc J. Lauret, E. Rodr\'\i guez-Valencia}, On the Chern-Ricci flow and its solitons for Lie groups, {\it Math. Nachrichten} 
{\bf 288} (2015), 1512–1526.

\bibitem[LW1]{einsteinsolv}  {\sc J. Lauret, C.E. Will}, Einstein solvmanifolds: existence and non-existence questions, {\it Math. Annalen} {\bf 350} (2011), 199-225.

\bibitem[LW2]{SCFmuA} {\sc J. Lauret, C. Will}, On the symplectic curvature flow for locally homogeneous manifolds, {\it J. Symp. Geom.}, in press (arXiv).

\bibitem[LW]{LtyWei} {\sc J. Lotay, Y. Wei}, Laplacian flow for closed $G_2$ structures: Shi-type estimates, uniqueness and compactness, preprint 2015 (arXiv).

\bibitem[M]{Mln} {\sc J. Milnor}, Curvature of Left-invariant Metrics on Lie Groups, {\it Adv. Math.} {\bf 21}(1976), 293-329.

\bibitem[N]{Ncl} {\sc M. Nicolini}, in preparation.

\bibitem[N1]{Nkl} {\sc Y. Nikolayevsky}, Einstein solvmanifolds and the pre-Einstein derivation, {\it Trans. Amer. Math. Soc.} {\bf 363} (2011), 3935-3958.

\bibitem[N2]{Nkl3} {\sc Y. Nikolayevsky}, Einstein solvmanifolds attached to two-step nilradicals, {\it Math. Z.} {\bf 272} (2012), 675-695.

\bibitem[N3]{Nkl2} {\sc Y. Nikolayevsky}, Solvable extensions of negative Ricci curvature of filiform Lie groups, {\it Math. Nachrichten}, in press (arXiv).

\bibitem[NN]{NklNkn} {\sc Y. Nikolayevsky, Y. Nikonorov}, Solvable Lie groups of negative Ricci curvature {\it Math. Z.} {\bf 280} (2015), 1-16. 

\bibitem[P1]{Pyn} {\sc T. Payne}, The Ricci flow for nilmanifolds, {\it J. Modern Dyn.} {\bf 4} (2010), 65-90.

\bibitem[P2]{Pyn2} {\sc T. Payne}, Applications of index sets and Nikolayevsky derivations to positive rank nilpotent Lie algebras, {\it J. Lie Theory} {\bf 24} (2014), 1-27.

\bibitem[Pe]{Ptr} {\sc P. Petersen}, Riemannian geometry, {\it GTM 171, Springer}  (1998).

\bibitem[Po]{Pk} {\sc J. Pook}, Homogeneous and locally homogeneous solutions to symplectic curvature flow, preprint 2012 (arXiv).

\bibitem[S]{Ssm} {\sc N. Sesum}, Curvature tensor under the Ricci flow, {\it Amer. J. Math.} {\bf 127} (2005), 1315-1324.

\bibitem[SW]{SngWnk} {\sc J. Song, B. Weinkove}, Lecture notes on the K\"ahler-Ricci flow, preprint 2012 (arXiv).

\bibitem[ST]{StrTn3} {\sc J. Streets, G. Tian}, Regularity results for pluriclosed flow, {\it Geom. Top.}, in press.

\bibitem[TW]{TstWnk} {\sc V. Tosatti, B. Weinkove}, On the evolution of a hermitian metric by its Chern-Ricci form, {\it J. Diff. Geom.} {\bf 99} (2015), 125-163.  

\bibitem[V]{Vzz2} {\sc L. Vezzoni}, A note on canonical Ricci forms on $2$-step nilmanifolds, {\it Proc. Amer. Math. Soc.} {\bf 141} (2013), 325-333.

\bibitem[WW]{WssWtt} {\sc H. Weiss, F. Witt}, Energy functionals and solitons equations for $G_2$-forms, {\it Ann. Glob. Anal. Geom.} {\bf 42} (2012), 585-610.

\bibitem[W1]{Wll1} {\sc C.E. Will}, Rank-one Einstein solvmanifolds of dimension $7$, {\it Diff. Geom. Appl.} {\bf 19} (2003), 307-318.

\bibitem[W2]{Wll2} {\sc C.E. Will}, A curve of nilpotent Lie algebras which are not Einstein nilradicals, {\it Monatsh. Math.} {\bf 159} (2010), 425-437.

\bibitem[W3]{Wll3} {\sc C.E. Will}, The space of solvsolitons in low dimensions, {\it Ann. Global Anal. Geom.} {\bf 40} (2011), 291-309.

\bibitem[W4]{Wll4} {\sc C.E. Will}, Negative Ricci curvature on some nonsolvable Lie groups, preprint 2015.
\end{thebibliography}
\end{document}